\documentclass{amsart}

\usepackage{epsfig} 
\usepackage{times} 
\usepackage{amsmath} 
\usepackage{amssymb}  
\usepackage{cases}
\usepackage{float}

\usepackage{mathtools}
\usepackage{latexsym}
\usepackage{dsfont}
\usepackage{graphicx}
\usepackage{float}
\usepackage{mathrsfs}
\usepackage{verbatim}
\usepackage{epstopdf}
\usepackage{url}
\usepackage[utf8]{inputenc}
\usepackage{hyperref}

\newtheorem{theorem}{Theorem}[section]
\newtheorem{lemma}{Lemma}[section]
\newtheorem{corollary}{Corollary}[section]
\newtheorem{proposition}{Proposition}[section]
\newtheorem{definition}{Definition}[section]
\newtheorem{example}{Example}

\newtheorem{remark}{Remark}[section]
\newtheorem{assumption}{Assumption}

\numberwithin{equation}{section}

\DeclareMathOperator\diag{diag}
\DeclareMathOperator\Tr{Tr}

\begin{document}

\title[Linear quadratic mean field social optimization]{Linear quadratic mean field social optimization:
Asymptotic solvability and decentralized control}

\author{Minyi Huang}
\address{School  of Mathematics and Statistics, Carleton University,
Ottawa, ON K1S 5B6, Canada}
\email{mhuang@math.carleton.ca}

\author{Xuwei Yang}
\address{School of Mathematics and Statistics, Carleton University, Ottawa, ON K1S 5B6, Canada}
\email{xuweiyang@cunet.carleton.ca}


\begin{abstract}
This paper studies asymptotic solvability of a linear quadratic (LQ) mean
field social optimization problem with controlled diffusions and indefinite state and control weights. Starting with an $N$-agent model,  we employ a rescaling approach to derive a low-dimensional Riccati ordinary differential equation (ODE) system, which characterizes
a necessary and sufficient condition for asymptotic solvability.
The  decentralized control obtained from the mean field limit ensures a bounded optimality loss
in minimizing the social cost having magnitude $O(N)$, which implies an
 optimality loss of $O(1/N)$ per agent. We further  quantify the efficiency  gain of the social optimum with respect to the solution of  the mean field game.
\end{abstract}

\maketitle

\tableofcontents

\section{Introduction}
\label{sec:intro}

In social optimization problems, multiple interacting agents have individual performance objectives but cooperatively optimize for the goal  of the whole group. For instance,  such scenarios arise in communication networks seeking network utility maximization, where the total utility of the users is maximized
\cite{CL2007,HTK2017,KMT1998}; the extension to the case of multi-period optimization can be found in \cite{HTK2017,TZB2008}.  Similarly, in the economic literature  social welfare functions have long been studied  \cite{MWG1995,M2004} as a key concept of welfare economics. A typical form
that they take is the sum of the individual utilities and is accordingly called a utilitarian social welfare function.

In this paper, we are concerned with
 mean field social optimization  involving $N$ agents which have mean field  coupling  through their individual  dynamics and costs and minimize a
social cost.
Consider a system of $N$ agents, denoted by $\mathcal{A}_i$, $1\leq i \leq N$. The state process $X_i(t)$ of $\mathcal{A}_i$ satisfies the following stochastic differential equation (SDE)
\begin{align}
\label{X_i}
d X_i(t) & = (A X_i(t) + Bu_i(t) + G X^{(N)}(t))  dt
 + (B_1 u_i(t) + D) d W_i(t)   \\
& \quad + ( B_0 u^{(N)}(t) + D_0 ) d W_0(t) , \notag
\end{align}
where we have the state $X_i(t) \in \mathbb{R}^n$, the control $u_i(t)\in
\mathbb{R}^{n_1}$,  the mean field state $X^{(N)} \coloneqq(1/N)\sum_{i=1}^N X_i$, and the control mean field  $u^{(N)} \coloneqq (1/N)\sum_{i=1}^N u_i$.
The initial states $\{X_i(0): 1 \leq i \leq N\}$ are independent with $\mathbb{E}X_i(0)= x_i(0)$ and $\mathbb{E}|X_i(0)|^2<\infty$. The individual noise processes $\{W_i: 1\leq i \leq N\}$ are $1$-dimensional independent standard Brownian motions, which are also independent of $\{X_i(0): 1\leq i \leq N\}$.
The common noise $W_0$ is a $1$-dimensional standard Brownian motion independent of $\{W_i: 1 \leq i \leq N\}$ and $\{X_i(0): 1 \leq i \leq N\}$.

The individual cost of agent ${\mathcal A}_i$, $1\leq i \leq N$, is given
by
\begin{align}
J_i(u_{1}, \cdots, u_N) & = \mathbb{E} \bigg[ \int_0^T \Big( [X_i(t) - \Gamma X^{(N)}(t)]_Q^2 + [u_i(t)]_{R}^2 \Big)   dt \nonumber \\
&\qquad +  [X_i(T) - \Gamma_f X^{(N)}(T)]_{Q_f}^2 \bigg] , \nonumber
\end{align}
where we denote the quadratic form $[y]_M^2 = y^T M y$ for a symmetric matrix $M$.
The social cost is defined as \begin{align}
J_{\rm soc}^{(N)}( u_{1}, \cdots, u_N)
&\coloneqq   \sum_{i=1}^N J_i  (u_{1}, \cdots, u_N  )  .
\label{J soc}
\end{align}
The constant matrices  $A$, $B$, $B_0$ $B_1$, $D$, $D_0$, $G$, $\Gamma$, $Q$, $R$, $\Gamma_f$ and $Q_f$ above have compatible dimensions, and $Q$,
$R$ and $Q_f$ are symmetric matrices.
The weight matrices $Q$, $R$ and $Q_f$ may be indefinite.
 LQ stochastic optimal control with indefinite control weights was first studied in \cite{CLZ1998}, which shows that the optimal control problem may still be well posed when the control enters the diffusion term. The more general case with indefinite state and control weight matrices are treated in \cite{RMZ2001,YZ1999}. It is shown in
\cite{RMZ2001} that the solvability of the stochastic optimal control problem is equivalent to the solvability of a generalized differential Riccati-type equation. For discrete-time LQ control problems with indefinite weight matrices, see
\cite{FN2015,RCZ2002}.

The above social optimization model differs from
mean field games in that the agents in the  latter are non-cooperative.
For general theory and applications of mean field games, the reader is referred to
\cite{BFY2013,CHM2017,C2013,CS2015,CD2018,GS2014,G2016,GLL2011,HMC2006,LL2007}.
LQ mean field games are a particularly attractive class of problems due to their explicit solutions \cite{BP2014,BSYY2016,HWW2016,HCM2007,MB2017,SMN2018}.

There has been a growing literature related to  mean field social optimization.  An LQ mean field social optimization problem has been considered in \cite{HCM2012}
with additive noise and positive definite  control weight and positive semi-definite state weight. That work constructs the limiting decision problems for the individual agents by use of the person-by-person (PbP) optimality principle
where a selected agent takes non-anticipative control perturbations. This
method is applied to  a nonlinear model in \cite{SHM2016}. The  work \cite{CH2019} studies social optimization with indefinite state weight.
Social optima are analyzed in \cite{WZ2017} for a mean field jump LQ  model governed by a common Markovian chain. An LQ social optimum model is studied in  \cite{SNM2018} for a large number of weakly coupled agents choosing cooperatively between multiple destinations.  A nonlinear social optimization problem for an infinite horizon economy is analyzed in \cite{NM2018}, where necessary conditions of the social optimum are derived by using G\^ateaux derivatives and Lagrangian multipliers treating market clearing
equality constraints.
A discrete-time LQ social optimization problem involving a finite number of  subsystems with mean-field state coupling is analyzed in \cite{AM2015} to obtain optimal control laws for both full observation and partial observation cases; this problem is called team-optimization to emphasize decentralized information structures. Further analysis of the mean field limit is developed in
\cite{AM2017}. Static mean field teams  with general costs are studied in
\cite{SY2020} under certain symmetry assumptions. It is shown that the solution obtained in  the limit problem has asymptotic optimality for the model with a finite number of agents. Mean field optimal control and flocking behavior of many interacting agents can be found in \cite{ACFK2017,FS2014}.

For a given $N$, the social optimum with the additive social cost  may be
viewed as a particular way of achieving a Pareto efficient solution in the sense that no individual can further improve for itself without causing
at least another agent to get worse. But Pareto optimality is a much weaker optimality notion and usually contains a set of qualified solutions. The reader may consult
\cite{E2010} for characterization of Pareto efficient solutions in cooperative differential games.
Both mean field games and mean field social optima are analyzed and compared in \cite{CR2019,LZZ2019}. The bounds for their efficiency difference are provided in \cite{CR2019} while \cite{LZZ2019} shows that the mean field equilibrium may be interpreted as the solution of a modified social optimization problem. Performance comparisons of the two solution approaches are presented in \cite{WWHS2021}
for a static mean field model arising in dense wireless networks.

\subsection{Related literature on mean field type optimal control}

It is worthwhile to mention the related area of mean field type optimal control problems which involve the state process together with its mean \cite{AD2011} or  its distribution (see e.g. \cite{BFY15,L2017}). Moreover, the model  involves only one decision maker, which immediately affects the distribution of the underlying state process. The optimal control is characterized by a stochastic maximum principle under a convex control set in \cite{AD2011}, and this approach is extended to deal with more general
dynamics and a possibly non-convex control set \cite{BLM2016}, which derives the maximum principle containing a second order adjoint equation. For LQ mean field type optimal control,  \cite{Y2013} derives the solution via
a system of forward-backward stochastic differential equations (FBSDEs) and further decouples the FBSDEs by Riccati equations to obtain the optimal control in an explicit form. For discrete-time mean field type control,
the reader is referred to \cite{ELN2013,NZL2015}.

Instead of just including mean terms in the model, the more general framework considers optimal control of McKean--Vlasov dynamics, where both the
system state and its law appear in the dynamics and costs. The optimal control law has the interpretation of a cooperative equilibrium in a large population model of $N$ agents coupled by the empirical distribution of their states \cite{CDL2013}, but the search for this cooperative equilibrium is based on the restriction that all agents use the same local feedback
control law $\varphi(t, X_i)$ to test optimality. The connection between large population social optimal control and optimal control of McKean--Vlasov dynamics is further addressed in \cite{L2017} and \cite{DPT20}. It is shown in \cite{L2017} that  the social optimal control problems of a large number of interacting state processes may be connected with optimal control problems of McKean--Vlasov type. Specifically, each relaxed optimal control of the McKean--Vlasov model may be obtained as the limit of  relaxed $\epsilon_N$-optimal controls for the $N$-agent  social optimal control problems, where $\epsilon_N\to 0$ as $N\to \infty$. Similar limit theorems are obtained in \cite{DPT20} for more general system dynamics together with  common noise, where the state equation uses a conditional law of the state-control pair given the common noise.  The dynamic programming approach is applied to mean field type optimal control in \cite{BFY15} to derive the so-called master equation.  For McKean--Vlasov optimal control problems with common noise, the dynamic programming principle is established in \cite{BCP2018,PW2017} by taking the distribution of the state as an abstract state subject to stochastic McKean--Vlasov dynamics.
An application to LQ optimal control problems is presented in \cite{P2016} dealing with positive (semi-)definite weight matrices.

While there is a close connection between large population social optimal
control and McKean--Vlasov optimal control (or mean field type control in
general) as analyzed in \cite{L2017,DPT20}, the two classes of models have crucial differences.
Firstly, the actual mechanisms affecting the mean field are different in that a single agent in the social optimization model has little impact on
the mean field.
Secondly, the McKean--Vlasov optimal control model typically assumes homogeneous agents while social optimization allows heterogeneity; for instance, the LQ model in \cite{HCM2012} allows the agents to have individual dynamic parameters varying from a continuum.
 Finally, the two classes of problems interpret  time consistency differently, and the social optimum may easily attain time consistency due to the particular mechanism generating the mean field (this point will be illustrated by  examples in Section \ref{sec:sub:compMFTC}).

\subsection{Our approach}

Our  study  of the model \eqref{X_i}--\eqref{J soc} deals with
control-dependent noises and indefinite control and state weight matrices.   More importantly, here we take a new perspective by adopting a notion called asymptotic solvability.
Roughly speaking, asymptotic solvability, which is formally defined in Section 3, is the solvability of the social optimization problem
\eqref{X_i}--\eqref{J soc} as $N\to\infty$.
Some early analysis has been presented in the conference paper \cite{HY2019a}.
The asymptotic solvability approach was initially developed in LQ mean field games \cite{HZ2018a,HZ2020}; that approach attempts to answer such a question for the games: Does there exist an intrinsic low-dimensional object that governs the large system's solution  generating a good asymptotic  behavior when the population size tends to infinity. The approach shares similarity with the convergence problem in the direct approach of mean field games \cite{CDLL15}. In \cite{HZ2018a,HZ2020}, a necessary and sufficient condition for asymptotic solvability of $N$-player LQ mean field games is obtained through analyzing a low-dimensional ODE system derived by applying a rescaling method to the sequence of high-dimensional centralized solutions.  Asymptotic solvability of LQ mean field games with a major player is studied in  \cite{MH2020}.

For the $N$-agent LQ optimal control problem \eqref{X_i}--\eqref{J soc} with indefinite weights, one in principle may use a Riccati equation to determine feedback optimal control \cite{SLY2016}. Due to the indefinite weights, the equation does not always have a solution, and determining the existence of a solution becomes a highly nontrivial task, especially when
$N$ is large.
This poses a conceptual obstacle before we can even think of deriving a mean field limit from a centralized solution and consider decentralized individual controls with little optimality loss. In this case, our formulation of the asymptotic solvability problem is particularly relevant for addressing this existence issue by coming up with a simple  criterion. Then our further analysis will  show that some simpler limiting objects (as two ODEs in a
lower-dimensional space) encodes all essential information for a well
behaved system when $N\to \infty$.
Specifically, our starting point here is to  apply dynamic programming to
derive the large-scale Riccati equation. This approach has several advantages for the present model over the person-by-person optimality argument in \cite{HCM2012}. First, the Riccati equation-based approach, as long as
its solvability holds, ensures optimality from the beginning. In contrast, the
PbP optimality-based approach is much harder to apply due to the common noise. Second, the large Riccati equation is particularly suitable  for applying the rescaling technique as in \cite{HZ2020}.
The LQ social optimization model in \cite{CH2018,HCM2012} involves
positive semi-definite state weight and positive definite control weight,
and the players have only independent noises.

Next, we determine  the closed-loop dynamics under the centralized  optimal control $U^o$ and use the mean field limit to derive a set of decentralized individual controls $U^d$ for which each agent uses only its own state and the mean field limit state. While the social optimum $J_{\rm soc}^{(N)}(U^o) $ has magnitude $O(N)$,
the decentralized control is shown to achieve bounded optimality loss with respect to the social optimum as $N\to \infty$, i.e., $0\le  J_{\rm soc}^{(N)}(U^d)-J_{\rm soc}^{(N)}(U^o)=O(1) $,  which is  tighter than the  upper bound  $O(\sqrt{N})$ for optimality loss obtained by the method
 in \cite{HCM2012}.
In \cite{AM2017} it is shown for a mean field team the mean field limit based policy can have an overall optimality loss of $O(1)$, but they consider
a relatively simple linear model with uncontrolled noise and do not face high nonlinearity of the Riccati equation. Their model has positive definite weights, and  asymptotic solvability automatically holds.

We also note that the limit theorems in \cite{L2017,DPT20} relate mean field social optimization to control of McKean--Vlasov dynamics. Their nonlinear models have much generality but the compactness-based analysis needs restrictive conditions on the control caused growth of the cost integrand.
These conditions cannot cover our indefinite quadratic cost.

\subsection{Contributions and organization}
 We extend the asymptotic solvability notion, initially introduced for mean field games, to social optimization. A key feature of our system is that the state and control weight matrices may be indefinite.
 Due to the highly nonlinear Riccati ODEs resulting from controlled diffusion terms, the development of the rescaling technique is more challenging than in  \cite{HZ2018a,HZ2020,MH2020}.
We further obtain a tight upper bound  of the optimality loss of the obtained decentralized controls, and quantify the efficiency gain with respect to mean field game solutions.

The paper is organized as follows.
In Section \ref{sec: lq mf soc}, we introduce the LQ mean field social optimization model and derive the large-scale Riccati equation for the optimal control.
 Section \ref{sec: AS}  introduces  the asymptotic solvability notion and
      presents a necessary and sufficient condition for asymptotic solvability via a low-dimensional Riccati ODE system.
Section~\ref{sec: cl dynamics}  gives the closed-loop state dynamics under the optimal control and its mean field limit.
Section~\ref{sec: performance analysis}  analyzes the associated decentralized  control by proving a bounded optimality gap result, and compares the performance with the mean field game solution.
This section also compares mean field social optimization with mean field type optimal control.
Section~\ref{sec: num} gives some numerical examples.
Section~\ref{sec: conclusion} concludes the paper.

\subsection{Notation}

We use $I$ to denote an identity matrix of compatible dimensions, and sometimes write $I_k$ to indicate the $k\times k$ identity matrix.
For a vector or matrix $F$, $|F|$ denotes the Euclidean norm of $F$.
For any $l\times m$ matrix $Z = (z_{ij})_{1\leq i \leq l, 1 \leq j \leq
m}$, we denote the $l_1$-norm $\left\| Z \right\|_{l_1}:= \sum_{i, j} |z_{ij}|$.
Let ${\mathcal S}^n$ be the set of $n\times n$ real symmetric matrices.
We denote by $\mathbf{1}_{k\times l}$  a $k\times l$ matrix with all entries equal to $1$, by $\otimes$ the Kronecker product, and by the column vectors $\{e^k_1, \cdots, e^k_k\}$ the canonical basis of $\mathbb{R}^k$.

\section{State feedback for LQ social optimization}
\label{sec: lq mf soc}

Define
\begin{align}
& X(t) = \begin{bmatrix}X_1(t) \\ \vdots \\ X_N(t) \end{bmatrix} \in \mathbb{R}^{Nn}, \quad
 U(t) = \begin{bmatrix} u_1(t) \\ \vdots \\ u_N(t) \end{bmatrix} \in \mathbb{R}^{N n_1}, \notag \\
&  \mathbf{A}
= \diag[A, \cdots, A] + \mathbf{1}_{N \times N} \otimes \tfrac{G}{N} \in \mathbb{R}^{Nn \times Nn} ,  \notag \\
&  \mathbf{B}_0 =  \mathbf{1}_{N\times 1} \otimes  \tfrac{B_0}{N} \in
\mathbb{R}^{Nn \times n_1} , \quad
\mathbf{D}_0 = \mathbf{1}_{N\times 1} \otimes D_0 \in \mathbb{R}^{Nn \times 1} ,   \notag \\
&  \widehat{\mathbf{B}}_k = e^N_k \otimes B  \in \mathbb{R}^{Nn \times n_1} , \quad
 \mathbf{B}_k = e^N_k \otimes B_1 \in \mathbb{R}^{Nn\times n_1} , \nonumber\\
&
 \mathbf{D}_k = e^N_k \otimes D \in \mathbb{R}^{Nn \times 1} , \quad 1 \leq k \leq N .   \notag
\end{align}
We write \eqref{X_i} in a compact form:
\begin{align}
d X(t) & =  \Big( \mathbf{A} X(t) + \sum_{i=1}^N \widehat{\mathbf{B}}_i u_i(t) \Big) dt + \sum_{i=1}^N ( \mathbf{B}_i u_i(t) + \mathbf{D}_i ) d W_i\label{X}\\
&\quad + \Big( \mathbf{B}_0 \sum_{i=1}^N u_i(t)  + \mathbf{D}_0 \Big) d W_0 .\nonumber
\end{align}
Define matrices:
\begin{align}
& \mathbf{Q}_1 =  \diag\left[Q, \cdots, Q \right] \in \mathbb{R}^{Nn\times Nn}, \qquad
\mathbf{Q}_2 = \mathbf{1}_{N\times N} \otimes
(Q^\Gamma /{N} ) \in \mathbb{R}^{Nn\times Nn} ,
\notag \\
& \mathbf{Q}_{1f} =  \diag\left[Q_f, \cdots, Q_f \right] \in \mathbb{R}^{Nn\times Nn},
\hspace{0.2cm}
\mathbf{Q}_{ 2f} = \mathbf{1}_{N\times N} \otimes (Q^\Gamma_f/{N} ) \in
\mathbb{R}^{Nn\times Nn} , \nonumber \\
&  \mathbf{Q} = \mathbf{Q}_1 + \mathbf{Q}_2  , \quad
  \mathbf{Q}_f = \mathbf{Q}_{1f} + \mathbf{Q}_{2f} ,
\hspace{0.46cm}
\mathbf{R} = \diag[R, \cdots, R] \in \mathbb{R}^{N n_1 \times N n_1} ,
\notag
\end{align}
where
\begin{align}
Q^\Gamma ={\Gamma^T Q \Gamma -Q \Gamma - \Gamma^T Q}, \qquad Q_f^\Gamma
={\Gamma_f^T Q_f \Gamma_f - Q_f \Gamma_f - \Gamma_f^T Q_f}.
\label{QGamma}
\end{align}

The social cost \eqref{J soc} may be rewritten as
\begin{align}
J_{\rm soc}^{(N)}(U) = & \mathbb{E} \bigg[ \int_0^T ( [ X(t) ]_\mathbf{Q}^2
+ [ U(t) ]_\mathbf{R}^2  ) dt
 + [X(T)]_{\mathbf{Q}_f}^2   \bigg] .
\label{J soc 2}
\end{align}

\subsection{The formal derivation of the Riccati equation}

Denote the  value function by $V(t, \mathbf{x})$ corresponding to the initial condition
$X(t) = \mathbf{x} = (x_1^T, \cdots, x_N^T)^T$ at time $t$.
The Hamilton--Jacobi--Bellman (HJB) equation of $V(t, \mathbf{x})$ is
\begin{align}
&  - \frac{\partial V}{\partial t} =
\min_{U\in \mathbb{R}^{N n_1}} \Big[ U^T \Big( \mathbf{R} + \mathcal{M}_2 \Big( \frac{\partial^2 V}{\partial \mathbf{x}^2}\Big) \Big) U + \Big( \frac{\partial^T V}{\partial \mathbf{x}} \widehat{\mathbf{B}} + \mathcal{M}_1 \Big(\frac{\partial^2 V}{\partial \mathbf{x}^2} \Big) \Big) U \Big]\label{HJB V}  \\
&\qquad\qquad\qquad\qquad+ \frac{\partial^T V}{\partial \mathbf{x}} \mathbf{A} \mathbf{x} + \mathbf{x}^T \mathbf{Q} \mathbf{x} + \mathcal{M}_0 \Big(\frac{\partial^2 V}{\partial \mathbf{x}^2} \Big)  ,
 \nonumber \\
& V(T, \mathbf{x})  = \mathbf{x}^T \mathbf{Q}_f \mathbf{x} , \notag
\end{align}
where we define the mappings
\begin{align}
& \mathcal{M}_0(Z) =    \frac{1}{2} \sum_{i=1}^N \mathbf{D}_i^T Z \mathbf{D}_i   + \frac{1}{2} \mathbf{D}_0^T Z \mathbf{D}_0 , \quad
 \mathcal{M}_1(Z) =  \sum_{i=1}^N \mathbf{D}_i^T Z \mathbf{B}_i
\mathbf{e}_i
 +  \mathbf{D}_0^T Z \mathbf{B}_0 \widehat{\mathbf{I}} , \notag \\
& \mathcal{M}_2(Z) =   \frac{1}{2} \sum_{i=1}^N \mathbf{e}_i^T \mathbf{B}_i^T Z \mathbf{B}_i \mathbf{e}_i + \frac{1}{2} \widehat{\mathbf{I}}^T
\mathbf{B}_0^T Z \mathbf{B}_0 \widehat{\mathbf{I}} ,
\quad  Z \in \mathbb{R}^{Nn\times Nn}  , \notag
\end{align}
which are from ${\mathbb R}^{Nn\times Nn}$ to $\mathbb{R}$, $\mathbb{R}^{1\times N n_1}$, and $\mathbb{R}^{Nn_1\times Nn_1}$, respectively,
and
\begin{align}
& \widehat{\mathbf{I}} = \mathbf{1}_{1\times N} \otimes I_{n_1} = \begin{pmatrix} I_{n_1}, \cdots, I_{n_1}   \end{pmatrix} \in \mathbb{R}^{n_1
\times N n_1} ,
\quad
 \widehat{\mathbf{B}} = (\widehat{\mathbf{B}}_1, \cdots, \widehat{\mathbf{B}}_N)
\in \mathbb{R}^{ Nn \times N n_1 } ,  \notag \\
& \mathbf{e}_i = ( e^N_i \otimes I_{n_1} )^T
= \begin{pmatrix} 0 ,  \cdots,  I_{n_1},  \cdots ,  0 \end{pmatrix}
\in \mathbb{R}^{n_1 \times N n_1} ,\quad
1\leq i \leq N .        \notag
\end{align}
The minimizer in \eqref{HJB V} is
\begin{align}
U = - \frac{1}{2} \Big( \mathbf{R} + \mathcal{M}_2 \Big(\frac{\partial^2 V}{\partial \mathbf{x}^2}\Big) \Big)^{-1} \Big(\frac{\partial^T V}{\partial \mathbf{x}} \widehat{\mathbf{B}} + \mathcal{M}_1 \Big(\frac{\partial^2 V}{\partial \mathbf{x}^2} \Big) \Big)^T ,
\label{U 1}
\end{align}
provided that $\mathbf{R} + \mathcal{M}_2( \tfrac{\partial^2 V}{\partial \mathbf{x}^2})$ is positive-definite.

We substitute the minimizer \eqref{U 1} into \eqref{HJB V} to obtain
\begin{align}
 - \frac{\partial V}{\partial t} = &
-  \frac{1}{4} \Big(\frac{\partial^T V}{\partial \mathbf{x}}
 \widehat{\mathbf{B}}
 +  \mathcal{M}_1\Big(\frac{\partial^2 V}{\partial \mathbf{x}^2}\Big)\Big)
\Big(\mathbf{R} + \mathcal{M}_2\Big(\frac{\partial^2 V}{\partial \mathbf{x}^2}\Big) \Big)^{-1}   \Big(\frac{\partial^T V}{\partial \mathbf{x}} \widehat{\mathbf{B}} + \mathcal{M}_1\Big(\frac{\partial^2 V}{\partial \mathbf{x}^2}\Big)\Big)^T  \notag  \\
& +  \mathcal{M}_0\Big(\frac{\partial^2 V}{\partial \mathbf{x}^2}\Big)
 + \frac{\partial^T V}{\partial \mathbf{x}} \mathbf{A} \mathbf{x} + \mathbf{x}^T \mathbf{Q} \mathbf{x}.
\label{HJB equation under optimal control}
\end{align}
Suppose $V(t, \mathbf{x})$ takes the following form
\begin{align}
V(t, \mathbf{x}) = \mathbf{x}^T \mathbf{P}(t) \mathbf{x} + 2 \mathbf{x}^T \mathbf{S}(t)  + {\bf r}(t) ,
\label{V ansatz}
\end{align}
where $\mathbf{P}$ is symmetric.
We substitute \eqref{V ansatz} into \eqref{HJB equation under optimal control} to derive the ODE system of $\mathbf{P}(t)$, $\mathbf{S}(t)$, and $\mathbf{r}(t)$:
\begin{align}
& \begin{cases}
\dot{\mathbf{P}}(t) = \mathbf{P} \widehat{\mathbf{B}}
 \left( \mathbf{R} + 2 \mathcal{M}_2( \mathbf{P} )  \right)^{-1} \widehat{\mathbf{B}}^T \mathbf{P}   -  \mathbf{P} \mathbf{A}
- \mathbf{A}^T \mathbf{P} - \mathbf{Q} , \\
\mathbf{P}(T) =  \mathbf{Q}_f , \quad
\mathbf{R} + 2 \mathcal{M}_2 (\mathbf{P}(t))  > 0 , \quad \forall t \in [0, T],
\end{cases} \label{ODE P} \\
& \begin{cases}\dot{\mathbf{S}}(t) = \mathbf{P} \widehat{\mathbf{B}} ( \mathbf{R} + 2 \mathcal{M}_2(\mathbf{P}) )^{-1} (\widehat{\mathbf{B}}^T \mathbf{S} + \mathcal{M}_1^T(\mathbf{P})) - \mathbf{A}^T \mathbf{S},
\\\mathbf{S}(T) = 0, \end{cases}     \label{ODE S} \\
 & \begin{cases}\dot{\mathbf{r}} (t) = ( \mathbf{S}^T \widehat{\mathbf{B}} + \mathcal{M}_1(\mathbf{P})) ( \mathbf{R} +  2 \mathcal{M}_2(\mathbf{P}) )^{-1}( \widehat{\mathbf{B}}^T \mathbf{S} +  \mathcal{M}_1^T(\mathbf{P}) )\\ \qquad\quad - 2 \mathcal{M}_0(\mathbf{P}) ,   \\\mathbf{r}(T) =
0.
\end{cases} \label{ODE r}
\end{align}

\begin{remark}\label{remark:Punique}
If $\mathbf{P}$ is a solution of the Riccati ODE \eqref{ODE P} on $[0,T]$, it is the unique solution. This holds since  the vector field of the ODE has a local Lipschitz property along the solution trajectory satisfying
$\mathbf{R} + 2 \mathcal{M}_2 (\mathbf{P}(t))  > 0$.
\end{remark}
\begin{remark}
\label{rmk: sol P S r}
If \eqref{ODE P} has a (unique) solution on $[0, T]$, then after substituting
$\mathbf{P}$, \eqref{ODE S} 
becomes a linear ODE  of $\mathbf{S}$ and has a unique solution on $[0,T]$.  We further uniquely solve
 ${\bf r}$  on $[0, T]$.
\end{remark}
The inverse matrix $( \mathbf{R} + 2 \mathcal{M}_2(\mathbf{P}))^{-1}$ involving
$\mathbf{P}$ results in high nonlinearity of the Riccati ODE \eqref{ODE P}.
This is due to the control dependent noises in \eqref{X_i}.

In the above, the HJB equation is used to provide  a formal derivation of
the ODE system of $(\mathbf{P}, \mathbf{S}, \mathbf{r})$.
The following theorem gives the optimal feedback control law $U^o(t)$ using the ODEs \eqref{ODE P}--\eqref{ODE S}. We can show the optimality of $U^o(t)$ by applying a completion-of-squares technique to the cost.

\begin{theorem}
\label{thm: U}
Suppose  that \eqref{ODE P} has a solution $\mathbf{P}$ on $[0, T]$.
Then we may uniquely solve \eqref{ODE S} and \eqref{ODE r}, and the social optimal control under the cost \eqref{J soc 2} is
\begin{align}
U^o(t) = -  ( \mathbf{R} + 2 \mathcal{M}_2(\mathbf{P}(t)) )^{-1}
[ \widehat{\mathbf{B}}^T ( \mathbf{P}(t) X(t) + \mathbf{S}(t) ) + \mathcal{M}_1^T(\mathbf{P}(t)) ] .
\label{U 2}
\end{align}
The optimal cost with the initial condition $(t, \mathbf{x})$ is given by
\eqref{V ansatz}.
\end{theorem}
\begin{proof}
The theorem follows from \cite[Theorem 6.6.1]{YZ1999}, \cite[Corollary 3.2]{RMZ2001} and Remark~\ref{rmk: sol P S r}.
\end{proof}

\section{Asymptotic solvability}
\label{sec: AS}

By Theorem~\ref{thm: U}, the Riccati ODE \eqref{ODE P} plays a central role in the study of the social optimization problem \eqref{X_i}--\eqref{J soc}. For this reason we start by analyzing
\eqref{ODE P}.

\begin{definition}
\label{def: AS}
The social optimization problem \eqref{X_i}--\eqref{J soc} has asymptotic
solvability (by feedback control) if there exists $N_0>0$ such that for all $N\geq N_0$, \eqref{ODE P} has a solution $\mathbf{P}$ on $[0, T]$ and
\begin{align}
& \sup_{N\geq N_0} \sup_{0\leq t \leq T}  \left\| \mathbf{P}(t) \right\|_{l_1}/N < \infty ,
\label{AS 2} \\
&\mathbf{R} + 2 \mathcal{M}_2 (\mathbf{P}(t))   \geq c_0 I, \quad \forall
N \geq N_0, \ \forall t \in [0, T] ,
\label{AS 3}
\end{align}
for some fixed constant $c_0 >0$.
\end{definition}

We give a heuristic argument for  making a correct guess of the factor $1/N$ required in \eqref{AS 2}. Consider the case $t=0$ with no noise. Let ${\mathcal A}_1$ be assigned the initial condition $X_1(0)=cx_0$, where $x_0\in \mathbb{R}^n$ is a unit vector and $c$ is a large constant. All other agents take zero initial states. By checking the $N$ individual costs, we have a rough upper bound  $O(c^2)$ for the optimal social cost, uniformly with respect to $N$. Recalling  \eqref{V ansatz},
for large $c$, the optimal cost is  $X^T(0){\mathbf P}(0)X(0)=c^2 x_0^T
P_{11}(0) x_0$, where the submatrix $P_{11}(0)$ is determined by the first $n$ rows and the first $n$ columns in $\mathbf{P}(0)$. Hence, we expect
to have
 $\|P_{11}(0)\|_{l_1}=O(1)$. The other $N-1$ diagonal submatrices in $\mathbf{P}(0)$  have the same bound by symmetry. The off-diagonal submatrices of dimension $n\times n$ are expected to have much smaller norm due to weak coupling. This suggests $\|{\mathbf P}(0)\|_{l_1}=O(N)$.

The test of asymptotic solvability by directly checking the sequence of $\mathbf{P}$ matrices is unfeasible due to the high nonlinearity and increasing dimensions of \eqref{ODE P}.
A central question is whether we can determine asymptotic solvability by some simple criterion.

\subsection{Main result}
\label{sec:sub:as}

For $\Lambda_1\in  {\mathcal S}^n$ and $\Lambda_2\in {\mathcal S}^n$,
define the mappings
\begin{align}
&{\mathcal R}_1(\Lambda_1)= R + B_1^T \Lambda_1 B_1, \\
& {\mathcal R}_2(\Lambda_1,\Lambda_2)=  R + B_1^T \Lambda_1 B_1 + B_0^T(\Lambda_1 + \Lambda_2 ) B_0.
\end{align}
Then ${\mathcal R}_1$ is from ${\mathcal S}^n $ to ${\mathcal S}^n$, and
 $ {\mathcal R}_2$ is from $ {\mathcal S}^n\times {\mathcal S}^n$   to
  ${\mathcal S}^n$.

Define the  ${\mathcal S}^n$-valued matrix functions
\begin{align}
 &\Psi_1(\Lambda_1) \coloneqq   \Lambda_1 B
( \mathcal{R}_1(\Lambda_1) )^{-1} B^T \Lambda_1 - \Lambda_1 A - A^T \Lambda_1 - Q ,
\label{psi_1}  \\
\label{psi_2}
& \Psi_2(\Lambda_1, \Lambda_2)
 \coloneqq  (\Lambda_1 + \Lambda_2 ) B
 ( \mathcal{R}_2(\Lambda_1, \Lambda_2)  )^{-1}
 B^T ( \Lambda_1 + \Lambda_2 )     \\
&  \qquad\qquad\qquad  - \Lambda_1 B ( \mathcal{R}_1(\Lambda_1) )^{-1} B^T \Lambda_1
 - \left[ \Lambda_1 G + \Lambda_2 (A + G) \right]
 \notag \\
&  \qquad\qquad\qquad  - \left[ G^T \Lambda_1 + (A^T + G^T) \Lambda_2 \right]
 - Q^\Gamma   ,  \notag
\end{align}
provided that each inverse matrix exists, where $\Lambda_1\in  {\mathcal S}^n$ and $\Lambda_2\in {\mathcal S}^n$.
The matrix $Q^\Gamma$ is specified in \eqref{QGamma}.

Denote the following ODE system
\begin{align}
& \begin{cases}
\dot \Lambda_1 (t) = \Psi_1(\Lambda_1(t)) , \\
\Lambda_1(T) = Q_f , \
 \mathcal{R}_1(\Lambda_1(t)) > 0 , \ \forall t \in [0, T] ,   \\
\end{cases}
\label{ODELam1} \\
& \begin{cases}
 \dot \Lambda_2(t)   =  \Psi_2(\Lambda_1(t), \Lambda_2(t) ) , \\
 \Lambda_2(T)   =Q^\Gamma_f   ,    \
\mathcal{R}_2(\Lambda_1(t) , \Lambda_2(t))  > 0, \ \forall t \in [0, T] .
  \\
\end{cases} \label{ODELam2}
\end{align}

If \eqref{ODELam1}--\eqref{ODELam2} has a solution on $[0,T]$, the solution is unique by similar reasoning as in Remark \ref{remark:Punique},
and both $\Lambda_1(t)$ and $\Lambda_2(t)$ are ${\mathcal S}^n$-valued.
The following theorem characterizes asymptotic solvability of the social optimization problem \eqref{X_i}--\eqref{J soc} in terms of the ODE system
\eqref{ODELam1}--\eqref{ODELam2}, which is a key result of this paper.
\begin{theorem}
\label{thm: NSAS}The social optimization problem \eqref{X_i}--\eqref{J soc} has asymptotic solvability
if and only if the ODE system \eqref{ODELam1}--\eqref{ODELam2} has a solution $(\Lambda_1, \Lambda_2)$ on $[0, T]$.
\end{theorem}

The rest of this subsection is devoted to proving Theorem~\ref{thm: NSAS}.

\begin{lemma}
\label{lm: P submatrices}
Suppose that \eqref{ODE P} has a solution $\mathbf{P}$ on $[0, T]$. Then
$\mathbf{P}$ has the representation
\begin{align}
\mathbf{P} = \begin{bmatrix} \Pi_1^N & \Pi_2^N &  \cdots & \Pi_2^N \\
\Pi_2^N & \Pi_1^N &  \cdots & \Pi_2^N \\
\vdots & \vdots & \ddots & \vdots \\
\Pi_2^N & \Pi_2^N & \cdots & \Pi_1^N  \end{bmatrix} ,
\label{P submatrices}
\end{align}
where both $\Pi_1^N(t)$ and $\Pi_2^N(t)$ are $n\times n$ symmetric matrix
functions of $t \in [0,T]$.
\end{lemma}
\begin{proof}
See Appendix~\ref{appendix: pf P submatrices}.
\end{proof}

\begin{lemma}
\label{lm: S submatrices}
Suppose that \eqref{ODE P} has a  solution $\mathbf{P}(t)$ on $[0, T]$. Then
 \eqref{ODE S} has a unique solution $\mathbf{S}$ on $[0, T]$ with  the representation
\begin{align}
\mathbf{S}(t) = (S^{NT}(t), \cdots, S^{NT}(t))^T \in \mathbb{R}^{Nn \times
1}, \quad S^N(t) \in \mathbb{R}^{n \times 1} .
\label{S submatrices}
\end{align}
\end{lemma}
\begin{proof}
See Appendix~\ref{appendix: pf S submatrices}.
\end{proof}

Intuitively, if we fix $x_i=x$ for all $i$, the value function $V(t, \mathbf{x})$ is expected to be of magnitude $O(N)$. On the other hand, \eqref{V ansatz} and \eqref{P submatrices} together give
\begin{align}
V(t, \mathbf{x}) = N x^T \Pi_1^N(t) x + (N^2-N) x^T \Pi_2^N(t) x  + 2N x^T S^N(t)
 + {\bf r}(t) = O(N) . \notag
\end{align}
This suggests we should have  $|\Pi_1^N(t)| = O(1)$, $|\Pi_2^N(t)| = O(1/N)$, and $|S^N(t)|=O(1)$ for any given $ t \in [0, T]$.
Based on the above heuristic reasoning on the  magnitude of
$|\Pi_1^N|$ and $|\Pi_2^N|$, we follow the rescaling method in
\cite{HZ2018a,HZ2020,MH2020} to define
\begin{align}
\Lambda_1^N(t) \coloneqq \Pi_1^N(t),
\quad \Lambda_2^N(t) \coloneqq N \Pi_2^N(t). \label{Lam12}
\end{align}
Then in view of Lemma \ref{lm: P submatrices}, $\mathbf{R} + 2 \mathcal{M}_2(\mathbf{P})$ may be denoted in  the form
\begin{align}
\mathbf{R} + 2 \mathcal{M}_2(\mathbf{P})
 =  \begin{bmatrix} F^N & K^N  & \cdots &  K^N  \\
 K^N  & F^N & \cdots &  K^N \\
\vdots & \vdots & \ddots & \vdots \\
 K^N & K^N & \cdots & F^N
  \end{bmatrix} ,
\label{R+2M_2(P)}
 \end{align}
where
\begin{align}
& K^N(t) = (1/N) {B_0^T}[  \Lambda_1^N (t)+ (1-1/N) \Lambda_2^N(t) ]{B_0} , \nonumber\\
&  F^N(t) =  K^N(t)+  \mathcal{R}_1(\Lambda_1^N(t))  . \notag
\end{align}

\begin{lemma}\label{lemma:chpoly}
Suppose $\mathbf{P}$ has a solution on $[0,T]$. Then given $t\in [0,T]$,
$\lambda$ is an eigenvalue of $ \mathbf{R} + 2 \mathcal{M}_2(\mathbf{P})$
if and only if either $\det[\lambda I - F^N - (N-1)K^N ] =0$ or
$\det[\lambda I - F^N +K^N ]=0$; moreover,
 for each $t$,
\begin{align}
 & \mathcal{R}_1(\Lambda_1^N(t))>0, \label{RL1ge0} \\
 &  \mathcal{R}_2(\Lambda_1^N(t), \Lambda_2^N(t)) - (1/N)B_0^T \Lambda_2^N(t) B_0 > 0. \label{RL12ge0}
\end{align}
\end{lemma}

\begin{proof}
The lemma follows from direct calculation of the characteristic polynomial
\begin{align}
&\det[ \lambda I - (\mathbf{R} + 2 \mathcal{M}_2 (\mathbf{P}) ) ] \nonumber \\
 =&  \det
 \begin{bmatrix} \lambda I - F^N & - K^N  &   \cdots & - K^N  \\
 - K^N  & \lambda I - F^N &  \cdots & - K^N \\
\vdots & \vdots & \ddots & \vdots  \\
 - K^N & - K^N   & \cdots & \lambda I - F^N
  \end{bmatrix} \nonumber  \\
 =& \det[\lambda I - F^N - (N-1)K^N ] \cdot(\det[\lambda I - F^N +K^N ])^{N-1}.  \notag
\end{align}
Note that both $F^N$ and $K^N$ are ${\mathcal S}^n$-valued.
The positive definiteness property in \eqref{RL1ge0}--\eqref{RL12ge0} follows from $\mathbf{R} + 2 \mathcal{M}_2(\mathbf{P})>0$  and $F^N-K^N> 0$,
 $F^N+(N-1) K^N > 0$.
\end{proof}

Since $\mathbf{R} + 2 \mathcal{M}_2(\mathbf{P})$ is symmetric, the inverse matrix
$(\mathbf{R} + 2 \mathcal{M}_2(\mathbf{P}))^{-1}$ also takes the following symmetric form
\begin{align}
(\mathbf{R} + 2 \mathcal{M}_2(\mathbf{P}) )^{-1}
=  \begin{bmatrix} H^N & E^N  & \cdots &  E^N  \\
 E^{N T}  & H^N &  \cdots &  E^N \\
\vdots & \vdots &  \ddots & \vdots \\
   E^{N T} &  E^{N T} &  \cdots & H^N
  \end{bmatrix} ,
\label{inv R+2M_2(P)}
\end{align}
where $E^N(t)$ and $H^N(t)$ are $n_1 \times n_1$ submatrices.
\begin{lemma}\label{lemma:EN}
The submatrix $E^N$ in \eqref{inv R+2M_2(P)} satisfies $E^N (t)= E^{N T}(t)$.
\end{lemma}
\begin{proof}
See Appendix~\ref{appendix: pf S submatrices}.
\end{proof}

By \eqref{R+2M_2(P)} and \eqref{inv R+2M_2(P)}, we get the relation
\begin{align}
& F^N H^N + (N-1) K^N E^N = I , \nonumber\\
& F^N E^N + K^N H^N + (N-2)K^N E^N  = 0 , \notag
\end{align}
which gives $ H^N  = E^N +  ( \mathcal{R}_1(\Lambda_1^N) )^{-1}$.
We obtain
\begin{align}
E^N= \ &(1/N)\{ [ \mathcal{R}_2(\Lambda_1^N, \Lambda_2^N)  - (1/N) B_0^T \Lambda_2^N  B_0]^{-1} -  ( \mathcal{R}_1(\Lambda_1^N) )^{-1}\},
\label{enform}\\
 H^N=\ & E^N+ ( \mathcal{R}_1(\Lambda_1^N) )^{-1},\label{hnform}
\end{align}
where each matrix inverse can be shown to  exist by Lemma \ref{lemma:chpoly}.

Our method below is to reduce the ODE of $\mathbf{P}(t)$ to some
lower-order ODE system.  We introduce the following system:
\begin{align}
& \begin{cases}
\dot \Lambda_1^N(t)  = \Psi_1(\Lambda_1^N )  + g_1(N, \Lambda_1^N, \Lambda_2^N) , \\
\dot \Lambda_2^N(t)  = \Psi_2(\Lambda_1^N , \Lambda_2^N) + g_2(N, \Lambda_1^N, \Lambda_2^N) ,  \\
\Lambda_1^N(T) = Q_f  + (1/N)Q^\Gamma_{f} , \quad
\Lambda_2^N(T)  = Q^\Gamma_f     ,  \\
 \mathcal{R}_1(\Lambda_1(t)) > 0 ,\quad
\  \mathcal{R}_2(\Lambda_1(t), \Lambda_2(t))    > 0 ,  \\
\mathcal{R}_2(\Lambda_1^N(t), \Lambda_2^N(t))   - ({1}/{N}) B_0^T \Lambda_2^N(t) B_0 > 0 ,
\quad \forall t \in [0, T] ,
\end{cases}
\label{ODELam12N}
\end{align}
where $g_1$ and $g_2$ are defined as
 \begin{align}
& g_1(N, \Lambda_1^N, \Lambda_2^N) \nonumber \\
&\coloneqq   \Lambda_1^N B E^N B^T \Lambda_1^N + (1-{1}/{N})( {\Lambda_2^N} B E^N  B^T \Lambda_1^N
  +  \Lambda_1^N B E^N B^T {\Lambda_2^N})
\notag \\
 &\quad  + (1/N-1/N^2){\Lambda_2^N} B [H^N + (N-2)E^N] B^T
 {\Lambda_2^N} \nonumber \\
& \quad-(1/N) [ (\Lambda_1^N  {G}+ {G^T}  \Lambda_1^N) + (1-1/N) ({\Lambda_2^N} {G}+ {G^T} {\Lambda_2^N})] - Q^\Gamma/N ,
\notag \\
& g_2(N, \Lambda_1^N, \Lambda_2^N) \nonumber  \\
&\coloneqq   (\Lambda_1^N + \Lambda_2^N) B \{ N E^N + (\mathcal{R}_1(\Lambda_1^N))^{-1}
 - (\mathcal{R}_2(\Lambda_1^N , \Lambda_2^N))^{-1} \} B^T (\Lambda_1^N + \Lambda_2^N ) \notag \\
& \quad - ({2}/{N})\Lambda_2^N B( \mathcal{R}_1(\Lambda_1^N))^{-1} B^T \Lambda_2^N  + ({1}/{N}-2) \Lambda_2^N B E^N B^T
 \Lambda_2 ^N \notag \\
& \quad - \Lambda_1^N B E^N B^T \Lambda_2^N - \Lambda_2^N B E^N B^T \Lambda_1^N
+ (\Lambda_2^N {G} + {G^T} \Lambda_2^N)/N , \notag
\end{align}
where  $E^N$ and $H^N$ are expressed as two functions of $(\Lambda_1^N, \Lambda_2^N)\in {\mathcal S}^n \times {\mathcal S}^n$ according to  \eqref{enform}--\eqref{hnform}.
How this system arises will be clear from our subsequent analysis. It is essentially derived from \eqref{ODE P} (which implies \eqref{RL1ge0}--\eqref{RL12ge0} ) after imposing the additional condition $\mathcal{R}_2(\Lambda_1^N, \Lambda_2^N)>0$ as required by $\Psi_2$ and $g_2$.

For the first term in the expression of $g_2$, we check
\begin{align}
\xi_N(\Lambda_1^N,\Lambda_2^N)\coloneqq &N E^N + (\mathcal{R}_1(\Lambda_1^N))^{-1}
  - ( \mathcal{R}_2(\Lambda_1^N, \Lambda_2^N) )^{-1}  \nonumber  \\
=& (1/N) (\mathcal{R}_2(\Lambda_1^N, (1-1/N)\Lambda_2^N))^{-1} B_0^T \Lambda_2^N  B_0 (\mathcal{R}_2(\Lambda_1^N, \Lambda_2^N))^{-1} . \label{g2mid}
\end{align}
And further  recalling the factor $1/N$ in the expression of $E^N$ in \eqref{enform}, we may view $g_1$ and $g_2$ as two small perturbation terms in the system \eqref{ODELam12N}.

\begin{remark}\label{remark:Lamode}
We have $\Psi_1: {\mathcal S}^n \to {\mathcal S}^n$, and
$\Psi_2,\ g_1(N, \cdot, \cdot),\ g_2(N, \cdot, \cdot): {\mathcal S}^n \times {\mathcal S}^n \to {\mathcal S}^n$.
The system \eqref{ODELam12N} may  stand alone without being immediately related to \eqref{P submatrices}.
\end{remark}

\begin{remark}
The third  positive-definiteness condition
in \eqref{ODELam12N} is needed due to the corresponding matrix inverse appearing in $E^N$, $g_1$ and $g_2$.
\end{remark}

The inverse matrix $(\mathbf{R} + 2 \mathcal{M}_2(\mathbf{P}))^{-1}$ in the Riccati ODE
\eqref{ODE P} contains submatrices $E^N$ and $H^N$, which are highly nonlinear in
$(\Lambda_1^N, \Lambda_2^N)$ according to \eqref{enform}--\eqref{hnform}.
Accordingly, \eqref{ODELam12N} is highly nonlinear.
This feature distinguishes our model from \cite{HZ2018a,HZ2020,MH2020}.

\begin{lemma} \label{lemma:P2Lam}
 \emph{(i)}
Suppose \eqref{ODE P} has a solution $\mathbf{P}$ on $[0,T]$, and let $(\Lambda_1^N, \Lambda_2^N)$ be defined by \eqref{P submatrices} and \eqref{Lam12}. Further assume $\mathcal{R}_2(\Lambda_1^N, \Lambda_2^N)>0$ for all $t\in [0,T]$. Then $(\Lambda_1^N, \Lambda_2^N)$ satisfies \eqref{ODELam12N} on $[0,T]$.

\emph{(ii)} Conversely, if $(\Lambda_1^N, \Lambda_2^N)$ is a solution of
\eqref{ODELam12N} on $[0,T]$, then \eqref{ODE P} has a (necessarily unique) solution $\mathbf{P}$ on $[0,T]$, which is related to $(\Lambda_1^N, \Lambda_2^N)$ by  \eqref{P submatrices} and \eqref{Lam12}.
 \end{lemma}

\begin{proof}
(i) After determining $(\Lambda_1^N, \Lambda_2^N)$  from $\mathbf{P}$ and
\eqref{Lam12}, it follows from  the last part of Lemma \ref{lemma:chpoly}
that the first and third inequality conditions in \eqref{ODELam12N} are satisfied.
By using \eqref{ODE P}, we  further derive the two ODEs in
\eqref{ODELam12N}.

(ii) Let $\mathbf{P}$ be defined by \eqref{P submatrices} and \eqref{Lam12} using $(\Lambda_1^N, \Lambda_2^N)$ solved from \eqref{ODELam12N}. By the characteristic polynomial in the proof of Lemma \ref{lemma:chpoly}, $\mathbf{R} + 2 \mathcal{M}_2(\mathbf{P})>0$ for all $t\in [0,T]$. Using the expression of  $(\mathbf{R} + 2 \mathcal{M}_2(\mathbf{P}))^{-1}$, we may directly verify the ODE in \eqref{ODE P}.
\end{proof}

\begin{lemma}\label{lemma:AStrue}
Suppose the social optimization problem \eqref{X_i}--\eqref{J soc}  has asymptotic solvability with $N\ge N_0$ in \eqref{AS 2}, and let $(\Lambda_1^N(t), \Lambda_2^N(t))$ be defined using $\mathbf{P}$  satisfying \eqref{ODE P}, \eqref{P submatrices}
and \eqref{Lam12}.  Then
there exists $N_1>N_{0}$ such that
$(\Lambda_1^N, \Lambda_2^N)$ satisfies \eqref{ODELam12N}  for all $N\ge N_1$ and  we further have
\begin{align}
& \sup_{N\geq N_1, 0\leq t \leq T}  (|\Lambda_1^N(t)| +  |\Lambda_2^N(t)|) < \infty ,  \label{lam12bnd}\\
& \mathcal{R}_1(\Lambda_1^N(t)) \ge c_1 I ,\quad \forall N\ge N_1,\label{Rc1I}\\
&  \mathcal{R}_2(\Lambda_1^N(t), \Lambda_2^N(t))    \ge c_1 I,\quad \forall N\ge N_1,\label{RB1c1}
\end{align}
for all $t\in [0,T]$, where $c_1>0$ is a fixed constant.
\end{lemma}

\begin{proof}
Suppose \eqref{AS 3} holds with the parameter $c_0$.
By the characteristic polynomial in the proof of Lemma \ref{lemma:chpoly},  we have
\begin{align}
 \mathcal{R}_1(\Lambda_1^N(t)) \ge c_0 I, \quad
 \mathcal{R}_2(\Lambda_1^N(t), \Lambda_2^N(t))  -
 (1/N) B_0^T \Lambda_2^N
B_0 \ge c_0 I \label{R1R2Nc0}
\end{align}
 for all $N\ge N_0$.
By \eqref{AS 2}  and the relation \eqref{P submatrices}, we have
$$
 \sup_{N\geq N_0, 0\leq t \leq T}  (|\Lambda_1^N(t)| +  |\Lambda_2^N(t)|)
< \infty.
$$
Hence there exists $N_1\ge N_0$ such that for all $N\ge N_1$, \eqref{Rc1I} and  \eqref{RB1c1} hold with $c_1=c_0/2$ by \eqref{R1R2Nc0}. Obviously \eqref{lam12bnd} holds. So for all $N\ge N_1$, \eqref{ODELam12N}  holds
by Lemma \ref{lemma:P2Lam} (i).
\end{proof}

\begin{lemma}
\label{lm: AS equiv}
Suppose there exists
$N_1>0$ such that
\eqref{ODELam12N} has a solution
$(\Lambda_1^N, \Lambda_2^N)$ on $[0, T]$  for all $N\geq N_1$, which further satisfies
\eqref{lam12bnd}--\eqref{RB1c1}
for some constant $c_1 >0$.
Then the social optimization problem \eqref{X_i}--\eqref{J soc} has asymptotic solvability.
\end{lemma}

\begin{proof}
First, after solving \eqref{ODELam12N} to obtain
$(\Lambda_1^N, \Lambda_2^N)$ for $N\ge N_1 $, let $\mathbf{P}$ be defined
by     \eqref{P submatrices} and \eqref{Lam12}. Then \eqref{ODE P} holds
by Lemma \ref{lemma:P2Lam} (ii).

By \eqref{lam12bnd} and \eqref{RB1c1},
there exists $N_2> N_1$ such that we have
\begin{align}
\zeta\coloneqq
\mathcal{R}_2(\Lambda_1^N, \Lambda_2^N) - ({1}/{N}) B_0^T \Lambda_2^N B_0\ge (c_1/2) I \notag
\end{align}
 for all $N\ge N_2$, $t\in [0,T]$.
Now for $N\ge N_2$,  by the proof of Lemma \ref{lemma:chpoly} all eigenvalues of
$\mathbf{R}+2\mathcal M_2(\mathbf{P})$ are exactly the solutions of the two equations
$$\det(\lambda I - \zeta)=0, \qquad
[\det(\lambda I - \mathcal{R}_1(\Lambda_1^N(t)))]^{N-1}=0.
$$
Hence $\mathbf{R}+2\mathcal M_2(\mathbf{P})\ge (c_1/2) I$. By \eqref{lam12bnd}, $\mathbf{P}$ satisfies \eqref{AS 2} by taking $N_0= N_2$. Therefore, asymptotic solvability holds.
\end{proof}

When there exists $N_1>0$ such that for each $N\geq N_1$,
\eqref{ODELam12N} has a solution $(\Lambda_1^N, \Lambda_2^N)$ on $[0, T]$
that satisfies \eqref{lam12bnd}--\eqref{RB1c1}, then by \eqref{enform} and \eqref{g2mid}
we obtain
\begin{align}
\sup_{0\le t\le T}| g_1(N, \Lambda_1^N, \Lambda_2^N)| = O(1/N) , \quad
  \sup_{0\le t\le T}|g_2(N, \Lambda_1^N, \Lambda_2^N)| =  O(1/N). \notag
\end{align}
The system \eqref{ODELam1}--\eqref{ODELam2} may be regarded as the limit of \eqref{ODELam12N}.
Lemmas~\ref{lemma:P2Lam} and \ref{lemma:AStrue} relate  asymptotic solvability of
the social optimization problem to the low-dimensional system
\eqref{ODELam12N}.

\begin{proof}[Proof of Theorem~\ref{thm: NSAS}]
(i)--Necessity.
If the social optimization problem \eqref{X_i}--\eqref{J soc} has asymptotic solvability,
by Lemma~\ref{lemma:AStrue}, there exists $N_1>0$ such that for each $N\geq N_1$, \eqref{ODELam12N} has a solution $(\Lambda_1^N, \Lambda_2^N)$
on $[0, T]$ that satisfies \eqref{lam12bnd}--\eqref{RB1c1} for some constant   $c_1 >0$.
From the integral form
\begin{align}
& \Lambda_1^N(t) = \Lambda_1^N(T) - \int_t^T [\Psi_1(\Lambda_1^N) + g_1(N, \Lambda_1^N, \Lambda_2^N)] d \tau ,  \label{int Lambda^N}  \\
& \Lambda_2^N(t) = \Lambda_2^N(T) - \int_t^T [\Psi_2(\Lambda_1^N, \Lambda_2^N) + g_2(N, \Lambda_1^N, \Lambda_2^N)] d\tau ,
\label{int Lambda_1^N}
\end{align}
we have that $\{ ( \Lambda_1^N(\cdot) , \Lambda_2^N(\cdot) ) \}_{N\geq N_1}$ are bounded and equicontinuous on $[0, T]$.
By Arzel\`{a}-Ascoli theorem~\cite{Y1980}, there exists a subsequence
$\{ ( \Lambda_1^{N_j}(\cdot) , \Lambda_2^{N_j}(\cdot) ) \}_{j\geq 1}$  that  converges to
$(\Lambda_1^\ast, \Lambda_2^\ast)$ uniformly on $[0, T]$ as $j\to\infty$.
Then it follows from \eqref{int Lambda^N}--\eqref{int Lambda_1^N} and \eqref{Rc1I}--\eqref{RB1c1} that
\begin{align}
& \Lambda_1^\ast(t) = \Lambda_1^\ast(T) - \int_t^T \Psi_1(\Lambda_1^\ast)  d\tau , \qquad \Lambda_2^\ast(t) = \Lambda_2^\ast(T) - \int_t^T \Psi_2(\Lambda_1^\ast, \Lambda_2^\ast )  d \tau ,  \quad
 \notag\\
& \mathcal{R}_1 ( \Lambda_1^\ast(t) ) \geq c_1 I , \quad
 \mathcal{R}_2 (\Lambda_1^\ast(t) , \Lambda_2^\ast(t) )
 \geq c_1 I , \quad \forall t \in [0, T] , \notag
\end{align}
where $\Lambda_1^\ast(T) = Q_f$ and $\Lambda_2^\ast(T)
= Q^\Gamma_f
$.
Thus $(\Lambda_1^\ast, \Lambda_2^\ast)$ solves the system \eqref{ODELam1}--\eqref{ODELam2}.

(ii)--Sufficiency.
Step 1.
 Suppose \eqref{ODELam1}--\eqref{ODELam2} has a  solution $(\Lambda_1, \Lambda_2)$ on $[0, T]$.
Then there exists $h_0 >0$ such that for all $ t\in [0,T]$, we have
\begin{align}
  \mathcal{R}_1(\Lambda_1(t)) \ge h_0 I, \quad
  \mathcal{R}_2(\Lambda_1(t), \Lambda_2(t)) \ge h_0 I.   \notag
\end{align}
We will check a neighborhood of the solution trajectory $(\Lambda_1, \Lambda_2)$ on $[0,T]$.
Since $(\Lambda_1, \Lambda_2)$ is continuous on $[0,T]$, there exists $\delta_0 >0$ such that
for all $(t, Z_1 , Z_2 )\in [0,T]\times {\mathcal  S}^n \times {\mathcal S}^n$ satisfying
$|Z_1 - \Lambda_1(t)| + | Z_2 - \Lambda_2(t) | < \delta_0$, we have
\begin{align}
\mathcal{R}_1(Z_1) \ge (h_0/2) I ,  \quad
\mathcal{R}_2(Z_1 , Z_2) \ge (h_0/2) I  .  \label{eigenh0}
\end{align}
Define
\begin{align}
\mathcal{C} & \coloneqq \{ (t, Z_1 , Z_2 ) \in  [0,T]\times {\mathcal S}^n \times {\mathcal S}^n : \
  | Z_1 - \Lambda_1(t) | + | Z_2 - \Lambda_2(t) | < \delta_0 \} . \notag
\end{align}
For the given $\delta_0$, there exists a sufficiently large $N_{\delta_0}$ such that  $N\ge N_{\delta_0}$ implies
\begin{align}
 \mathcal{R}_2(Z_1 , Z_2)-(1/N) B_0^TZ_2B_0 \ge (h_0/4)I \label{R212h0}
 \end{align}
 for all $(t, Z_1, Z_2)\in {\mathcal C}$.
By \eqref{eigenh0} and boundedness of $\mathcal{C}$, there exist  constants $L_\Psi$ and $C_g$ depending on ${\mathcal C}$ but not on $N$ such that   for all $(t, Z_1, Z_2)\in {\mathcal C}$
and  all $(t, Z_1', Z_2')\in {\mathcal C}$,  we have
\begin{align*}
&|\Psi_1(Z_1)-\Psi_1(Z_1')|
+ |\Psi_2(Z_1,Z_2)-\Psi_2(Z_1',Z_2')|
\le\  L_\Psi(|Z_1-Z_1'|+|Z_2-Z_2'| ),
\end{align*}
and moreover,
$| g_1(N,Z_{1}, Z_2)|+ |g_2(N,Z_1, Z_2)|\le C_g/N$ holds for all $ N\ge N_{\delta_0}$
in view of \eqref{enform}, \eqref{g2mid}, \eqref{eigenh0} and \eqref{R212h0}.

Step 2.
Consider
\eqref{ODELam12N} (see Remark \ref{remark:Lamode}).
Since
\begin{align}
\lim_{N\to\infty} ( |\Lambda_1^N(T) - \Lambda_1(T)| + |\Lambda_2^N(T) - \Lambda_2(T)| ) = 0,\label{LaLaT0}
\end{align}
there exists $N_1\ge N_{\delta_0}$ such that for all $N\ge N_1$, we have
\begin{align}
& |\Lambda_1^N(T) - \Lambda_1(T) | + |\Lambda_2^N(T) - \Lambda_2(T) | < \delta_0/2 ,  \notag \\
&   \mathcal{R}_1(\Lambda_1^N(T))    \geq c I , \quad   \mathcal{R}_2(\Lambda_1^N(T), \Lambda_2^N(T))   \geq c I  ,
\notag\\
& \mathcal{R}_2(\Lambda_1^N(T), \Lambda_2^N(T)) -(1/N) B_0^T \Lambda_2^N(T) B_0  \geq c I ,
\end{align}
where $c>0$ is a constant.
Then for each $N\ge N_1$, the solution $(\Lambda_1^N, \Lambda_2^N)$ in \eqref{ODELam12N}   exists on some interval $[t_N , T]$, with $0 \leq t_N < T$.

Step 3.
Our plan  is to show that there exists a sufficiently large $N_2>N_1$  chosen in Step 2 such that for all $N\ge N_2$,  \eqref{ODELam12N} has a solution on $[0, T]$.

 By \eqref{LaLaT0}, we may fix a sufficiently large $\hat N\ge N_1$ such that  $N\ge \hat N$ implies
\begin{align}
 \left( | \Lambda_1^{N}(T) - \Lambda_1(T) | + |\Lambda_2^N(T) - \Lambda_2(T) | + {C_g T}/{N} \right) \exp(L_\Psi T)\le \delta_0/2.\label{LLdel0}
\end{align}
Now it suffices to show that there exists a sufficiently large $N_2 \ge N_1$ such that for all $N\ge N_2$, we have
\begin{align}
| \Lambda_1^N(t) - \Lambda_1(t)| + | \Lambda_2^N(t) - \Lambda_2(t)| < \delta_0 , \quad \forall t \in [0, T]  ,   \label{LLdel0T}
\end{align}
which then implies that $(\Lambda_1^N, \Lambda_2^N)$ exists on $[0, T]$ by \eqref{eigenh0} and \eqref{R212h0}.
Assume by contradiction that given any $l >\hat N$ there always exists some $ N^* \geq l$  such that $(t, \Lambda_1^{N^*}(t), \Lambda_2^{N^*}(t) )$
starting backward from the terminal time $T$ exits ${\mathcal C}$ for the
first time at some $t^{N^*}_0 \in [0, T)$, i.e.,
\begin{align}
[0,T) \ni t^{N^*}_0 = \sup \{ t\in [0, T]:
(t, \Lambda_1^{N^*}(t), \Lambda_2^{N^*}(t) ) \notin \mathcal{C} \},\label{tonT}
\end{align}
where $t^{N^*}_0$ may depend on $N^*$.
Since
\begin{align}
 | \Lambda_1^{N^*}(t) - \Lambda_1(t) | + | \Lambda_2^{N^*}(t) - \Lambda_2(t) | \leq \delta_0 , \quad \forall t \in [t^{N^*}_0 , T] , \notag
\end{align}
it follows that $|\Lambda_1^{N^*}|$ and $|\Lambda_2^{N^*}|$ are bounded on
$[t^{N^*}_0 , T]$. On $[t^{N^*}_0 , T]$, by Step 1  we have
\begin{align*}
&|\Psi_1(\Lambda_1^{N^*}(t))-\Psi_1(\Lambda_1(t))|
+ |\Psi_2(\Lambda_1^{N^*}(t) , \Lambda_2^{N^*}(t))-\Psi_2(\Lambda_1(t)  ,\Lambda_2(t) )|\\
& \qquad \le  L_\Psi(|\Lambda_1^{N^*}(t)-\Lambda_1(t)  |+|\Lambda_2^{N^*}(t)-\Lambda_2(t)| ),\\
&| g_1(N^*, \Lambda_1^{N^*} (t), \Lambda_2^{N^*}(t))|+ |g_2(N^*, \Lambda_1^{N^*}(t) , \Lambda_2^{N^*}(t)|\le C_g/N^*
\end{align*}
since $N^*\ge N_{\delta_0}$.

It then follows that for any $t\in [t^{N^*}_0, T]$,
\begin{align}
 & |\Lambda_1^{N^*}(t) - \Lambda_1(t) | + | \Lambda_2^{N^*}(t)  -  \Lambda_2(t) |
\notag \\
 & \leq  | \Lambda_1^{N^*}(T) - \Lambda_1(T)  | + |\Lambda_2^{N^*}(T) -
 \Lambda_2(T)|\nonumber \\
& + \int_{t}^T \Big[| \Psi_1(\Lambda_1^{N^*}) - \Psi_1(\Lambda_1 ) |
+ | \Psi_2( \Lambda_1^{N^*}, \Lambda_2^{N^*}) - \Psi_2(\Lambda_1 , \Lambda_2) |\Big] d \tau  \notag \\
& \qquad  + \int_{t}^T \Big[| g_1 (N^*, \Lambda_1^{N^*}, \Lambda_2^{N^*} ) |
 + |g_2 (N^* , \Lambda_1^{N^*}, \Lambda_2^{N^*})|\Big] d \tau
  \notag \\
 & \leq  |\Lambda_1^{N^*}(T) - \Lambda_1(T) | + |\Lambda_2^{N^*}(T) -
  \Lambda_2(T)| \nonumber  \\
& + \int_{t}^T L_\Psi (  |\Lambda_1^{N^*}  - \Lambda_1 | + |\Lambda_2^{N^*}  - \Lambda_2 | )  d\tau  + \int_{0}^T \frac{C_g}{N^*} d \tau  .
\notag
\end{align}
By Gr\"{o}nwall's lemma, we have that for all $t \in [t^{N^*}_0, T]$,
\begin{align}
 & | \Lambda_1^{N^*}(t) - \Lambda_1(t) | + | \Lambda_2^{N^*}(t) - \Lambda_2(t)|
 \notag \\
& \leq  \left( | \Lambda_1^{N^*}(T) - \Lambda_1(T) | + |\Lambda_2^{N^*}(T) - \Lambda_2(T) | + {C_g T}/{N^*} \right) \exp(L_\Psi T) ,  \notag
\end{align}
which combined with \eqref{LLdel0} implies that
$$
\sup_{t\in [t_0^{N^*}, T]}( | \Lambda_1^{N^*}(t) - \Lambda_1(t) | + | \Lambda_2^{N^*}(t) - \Lambda_2(t)|)\le \delta_0/2.
$$
This contradicts  the hypothesis in \eqref{tonT} that $(t, \Lambda_1^{N^*}(t), \Lambda_2^{N^*}(t) )$ exits $\mathcal{C}$  at $t^{N^*}_0$. Hence, there exists $N_2>N_1$ such that for all $N\ge N_2$,  \eqref{LLdel0T} holds so that $(\Lambda_1^N, \Lambda_2^N)$  exists on $[0,T]$.
In view of  \eqref{eigenh0}, we further obtain
\begin{align*}
  \mathcal{R}_1(\Lambda_1^N(t)) \ge (h_0/2) I, \quad
  \mathcal{R}_2(\Lambda_1^N(t), \Lambda_2^N(t)) \ge (h_0/2) I
\end{align*}
 for all $N\ge N_2$
and  all $t\in[0, T]$.
Then by Lemma \ref{lm: AS equiv}, the social optimization problem has asymptotic solvability.
\end{proof}

\begin{corollary}
\label{cor: sol Lambda_1^N Lambda_2^N}
If \eqref{ODELam1}--\eqref{ODELam2} has a  solution $(\Lambda_1, \Lambda_2)$ on $[0, T]$, then there exists $N_1 >0$ such that for each $N\geq N_1$, \eqref{ODELam12N} has a solution $(\Lambda_1^N, \Lambda_2^N)$ on $[0, T]$ and moreover
$\sup_{t\in [0, T]} ( |\Lambda_1^N(t) - \Lambda_1(t)| + |\Lambda_2^N(t) -
\Lambda_2(t)| ) = O(1/N)$.
\end{corollary}
\begin{proof}
By Theorem~\ref{thm: NSAS} and Lemma \ref{lemma:AStrue},
if \eqref{ODELam1}--\eqref{ODELam2} has a solution $(\Lambda_1, \Lambda_2)$ on $[0, T]$, then there exists
$N_1 >0$ such that for each $N\geq N_1$, \eqref{ODELam12N} has a solution
$(\Lambda_1^N, \Lambda_2^N)$ on $[0, T]$ that satisfies
\eqref{lam12bnd}--\eqref{RB1c1}. 
 Then there exists a constant $L_1>0$ such that for all $N\ge N_1$ and for all $t\in [0,T]$, we have
\begin{align}
 & |\Psi_1(\Lambda_1^N) - \Psi_1(\Lambda_1)| \leq L_1|\Lambda_1^N - \Lambda_1 | ,
 \nonumber  \\
& |\Psi_2(\Lambda_1^N, \Lambda_2^N) - \Psi_2 (\Lambda_1, \Lambda_2)|
\leq L_1|\Lambda_1^N - \Lambda_1 | , \notag \\
& |g_1(N, \Lambda_1^N, \Lambda_2^N)| \leq  {L_1}/{N}, \quad
|g_2(N, \Lambda_1^N, \Lambda_2^N ) | \leq {L_1}/{N} . \notag
\end{align}

So combining \eqref{ODELam12N} and \eqref{ODELam1}--\eqref{ODELam2}, we obtain
\begin{align}
| \Lambda_1^N(t) - \Lambda_1(t) |
\leq\ & |Q^\Gamma_f    |/N
 +  \int_t^T L_1( | \Lambda_1^N(s) - \Lambda_1(s) |  + {1}/{N} ) ds ,
 \notag \\
 | \Lambda_2^N(t) - \Lambda_2(t) | \leq & \int_t^T
L_1 ( |\Lambda_1^N(s) - \Lambda_1(s) | + {1}/{N}  ) d s
\notag
\end{align}
for all $N\geq N_1$,  all $t\in [0, T]$.
By Gr\"{o}nwall's lemma,
 the desired result  follows.
\end{proof}

Let $(\Lambda_1^N, \Lambda_2^N)$ be given by  \eqref{ODELam12N}.
 We  further introduce the following ODE system
\begin{align}
& \begin{cases}
\dot{S}^N(t) =  \varphi_1 (\Lambda_1^N, \Lambda_2^N, S^N)
+ g_{01}(N, \Lambda_1^N, \Lambda_2^N, S^N) ,  \\
S^N(T) = 0 ,
\end{cases} \label{ODESN} \\
& \begin{cases}
\dot{r}^N(t)=  \varphi_2 (\Lambda_1^N, \Lambda_2^N, S^N)
+  g_{02} (N, \Lambda_1^N, \Lambda_2^N, S^N) ,  \\
 r^N(T) = 0 ,
\end{cases} \label{ODErN}
\end{align}
where
\begin{align}
\varphi_1(\Lambda_1^N, \Lambda_2^N, S^N )
\coloneqq &
( \Lambda_1^N +  \Lambda_2^N )  B (  \mathcal{R}_2(\Lambda_1^N, \Lambda_2^N) )^{-1}  \cdot   \notag \\
 & [ B^T S^N + B_1^T \Lambda_1^N D + B_0^T  ( \Lambda_1^N +  \Lambda_2^N
) D_0  ]
  - (A+G)^T S^N ,    \notag\\
 \varphi_2(\Lambda_1^N, \Lambda_2^N, S^N)
 \coloneqq &
 [S^{N T} B + D^T \Lambda_1^N B_1 + D_0^T ( \Lambda_1^N +  \Lambda_2^N ) B_0 ]
  (  \mathcal{R}_2(\Lambda_1^N, \Lambda_2^N) )^{-1}   \cdot \notag \\
 &   [ B^T S^N + B_1^T \Lambda_1^N D + B_0^T ( \Lambda_1^N
 +  \Lambda_2^N  ) D_0 ] \notag \\
& -  D^T \Lambda_1^N D -   D_0^T ( \Lambda_1^N + \Lambda_2^N ) D_0 ,
\notag
\end{align}
\begin{align}
 g_{01} (N, \Lambda_1^N, \Lambda_2^N, S^N) := &
 [ \Lambda_1^N + (1-1/N) \Lambda_2^N ]  B
   [ \mathcal{R}_2 (\Lambda_1^N, \Lambda_2^N) - (1/N) B_0^T \Lambda_2^N B_0 ]^{-1}  \cdot \notag \\
&  \Big\{ B^T S + B_1^T \Lambda_1^N D + B_0^T  [ \Lambda_1^N + (1-1/N) \Lambda_2^N  ] D_0  \Big\} \notag \\
& -   ( \Lambda_1^N +  \Lambda_2^N )  B (  \mathcal{R}_2(\Lambda_1^N, \Lambda_2^N) )^{-1}  \cdot \notag \\
 &  [ B^T S + B_1^T \Lambda_1^N D + B_0^T  ( \Lambda_1^N +  \Lambda_2^N  ) D_0  ] , \notag
\end{align}
\begin{align}
g_{02}(N, \Lambda_1^N, \Lambda_2^N , S^N) := &
  [S^{N T} B + D^T \Lambda_1^N B_1 + D_0^T ( \Lambda_1^N + (1-1/N) \Lambda_2^N ) B_0 ] \cdot  \notag \\
 &   [ \mathcal{R}_2 (\Lambda_1^N, \Lambda_2^N) - (1/N) B_0^T \Lambda_2^N
B_0 ]^{-1} \cdot \notag \\
 &   [ B^T S^N + B_1^T \Lambda_1^N D + B_0^T
 (\Lambda_1^N +  (1-1/N) \Lambda_2^N  ) D_0 ]  \notag \\
&  -  [S^{N T} B + D^T \Lambda_1^N B_1 + D_0^T ( \Lambda_1^N +  \Lambda_2^N ) B_0 ]
   ( \mathcal{R}_2(\Lambda_1^N, \Lambda_2^N))^{-1}   \cdot \notag \\
 &   [ B^T S^N + B_1^T \Lambda_1^N D + B_0^T (\Lambda_1^N
 +  \Lambda_2^N  ) D_0 ]
 + (1/N) D_0^T \Lambda_2^N D_0 .
\notag
\end{align}
The above ODEs are constructed by substituting \eqref{S submatrices} into
\eqref{ODE S} and writing  $ \mathbf{r}\coloneqq N  r^N $ in \eqref{ODE r}.

\begin{remark}
\label{rmk: sol S r}
If \eqref{ODELam1}--\eqref{ODELam2} has a solution $(\Lambda_1, \Lambda_2)$ on $[0,T]$, then the following system
\begin{align}
\label{ODESlim}
 \dot{S} = & \varphi_1( \Lambda_1, \Lambda_2, S) , \quad S(T) = 0 ,  \\
 \dot{r} = & \varphi_2 (\Lambda_1, \Lambda_2, S ) , \quad r(T) = 0 ,
\label{ODErlim}
\end{align}
 admits a unique solution $(S, r)$ on $[0, T]$.
\end{remark}

\begin{corollary}
\label{cor: sol S^N r^N}
If \eqref{ODELam1}--\eqref{ODELam2} has a solution $(\Lambda_1, \Lambda_2)$ on $[0, T]$, then there exists $N_1 >0$ such that  for all $N\geq N_1$,
\emph{(i)} the system \eqref{ODESN}--\eqref{ODErN} admits a unique solution $(S^N, r^N)$; \emph{(ii)} with $(S^N, r^N)$ determined in \emph{(i)},
$\mathbf{S}$ defined by \eqref{S submatrices}
and $\mathbf{r}\coloneqq Nr^N$ give a solution to \eqref{ODE S}--\eqref{ODE r}; and \emph{(iii)}
  $\sup_{t\in [0, T]} ( |S^N - S| + |r^N - r| ) = O(1/N)$, where $(S, r)$ is the  solution of \eqref{ODESlim}--\eqref{ODErlim}.
\end{corollary}
\begin{proof}
(i)  If \eqref{ODELam1}--\eqref{ODELam2} has a solution, it  follows from
Corollary~\ref{cor: sol Lambda_1^N Lambda_2^N} that there exists $ N_1>0$
such that for all $N\geq N_1$,
\eqref{ODELam12N} has a unique solution $(\Lambda_1^N, \Lambda_2^N)$. Substituting  $(\Lambda_1^N, \Lambda_2^N)$ into
\eqref{ODESN}--\eqref{ODErN} gives a first order linear ODE system of $(S^N, r^N)$ that admits a unique solution.

(ii) By substituting  $\mathbf{S}$ and $\mathbf{r}$ defined as above into
\eqref{ODE S}--\eqref{ODE r}, we may directly verify the ODE system \eqref{ODE S}--\eqref{ODE r}.

 (iii) The proof follows  similar steps as in the proof of Corollary~\ref{cor: sol Lambda_1^N Lambda_2^N}, and we omit the details.
\end{proof}

\subsection{Solvability of the limiting ODE system}
\label{sec:sub:solLam}

Determining the solvability of the ODE system
\eqref{ODELam1}--\eqref{ODELam2} on $[0,T]$ is  an interesting problem.
For this subsection, the analysis is restricted to the case $Q\ge 0, Q_f\ge 0$ while
the matrix $R$ may be indefinite.

Define $\Lambda_3  \coloneqq \Lambda_1 + \Lambda_2$.  Adding up both sides of
\eqref{ODELam1} and \eqref{ODELam2}, we   obtain the  Riccati ODE:
\begin{align}
& \begin{cases}
\dot {\Lambda}_3(t)  = \Psi_3(\Lambda_1, \Lambda_3)  , \
 R + B_1^T \Lambda_1(t) B_1 + B_0^T \Lambda_3(t) B_0  > 0 , \
\forall t \in [0, T], \\
 \Lambda_3(T)  = Q_{3f} , \\
\end{cases}  \label{ODELam3}
\end{align}
where $\Psi_3$ and $Q_{3f}$ are defined as
\begin{align}
& \Psi_3(\Lambda_1, \Lambda_3) =  \Lambda_3 B ( R + B_1^T \Lambda_1 B_1
+ B_0^T \Lambda_3 B_0  )^{-1} B^T \Lambda_3 \notag\\
&\qquad\qquad\qquad - \Lambda_3 (A + G) - (A + G )^T \Lambda_3  - Q_3 , \notag    \\
& Q_3
= (I-\Gamma)^T Q (I-\Gamma) , \quad
  Q_{3f} =(I-\Gamma_f)^T Q_f (I-\Gamma_f)  . \notag
\end{align}
Since the transformation $(\Lambda_1, \Lambda_2) \to (\Lambda_1, \Lambda_3)$ is one-to-one,
\eqref{ODELam1}--\eqref{ODELam2} has a unique solution on $[0,T]$ if and only if  the system consisting of \eqref{ODELam1} and \eqref{ODELam3} has a unique solution on $[0,T]$.

We consider existence and uniqueness of the solution of \eqref{ODELam1} and \eqref{ODELam3}. Denote ${\mathcal S}_{1}^{n_1}=\{Z: Z\in {\mathcal S}^{n_1}, Z>0\}$.
According to \cite[Theorem 4.6]{CLZ1998}, under the condition $Q\geq 0$ and $Q_f \geq 0$, the Riccati equation \eqref{ODELam1} admits a solution on $[0,T]$ if and only if there exists a function $K\in C([0,T], {\mathcal S}^{n_1}_1)$ such that
$R + B_1^T \widetilde\Lambda_1 B_1 \geq K$,
where $\widetilde\Lambda_1$ is the unique solution of the  standard Riccati ODE with $K$ taken as a parameter:
\begin{align}
\begin{cases}
\dot{\widetilde\Lambda}_1 (t) = \widetilde\Lambda_1 B K^{-1} B^T
\widetilde{\Lambda}_1 - \widetilde{\Lambda}_1 A - A^T \widetilde{\Lambda}_1 - Q, \\
\widetilde{\Lambda}_1(T) = Q_f .
\end{cases}
\label{ODEtildeLam1}
\end{align}
According to \cite{W1968},
\eqref{ODEtildeLam1} with positive definite $K$ has a unique solution on $[0, T]$.

Once $\Lambda_1$ is solved from \eqref{ODELam1}, we continue to solve   \eqref{ODELam3} as a Riccati ODE with the time-varying coefficient $R+B_1^T\Lambda_1 B_1$ to determine
$\Lambda_3$.
By \cite[Theorem 4.6]{CLZ1998}, for $Q\geq 0$ and $Q_{f}\geq 0$, \eqref{ODELam3} admits a solution if and only if there exists $K\in C([0,T], {\mathcal S}^{n_1}_1)$ such that
\begin{align}
R + B_1^T \Lambda_1 B_1 + B_0^T \widetilde\Lambda_3 B_0 \geq K , \notag
\end{align}
where $\widetilde{\Lambda}_3$ is the unique solution of the following standard Riccati ODE:
\begin{align}
\begin{cases}
 \dot {\widetilde \Lambda}_3(t)
 = \widetilde \Lambda_3B K^{-1} B^T\widetilde \Lambda_3 - \widetilde \Lambda_3 (A + G) - (A^T + G^T) \widetilde \Lambda_3
  - Q_3 ,    \\
 \widetilde{\Lambda}_3(T) =  Q_{3f} .
\end{cases}  \label{ODE tilde Lambda_3}
\end{align}

\subsection{Interpretation of the limiting Riccati ODEs}

We relate the  Riccati equations
\eqref{ODELam1}--\eqref{ODELam2} to  optimal control problems in a
low-dimensional space.
Consider a single-agent optimal control problem with state $X_1$ that satisfies
\begin{align}
d X_1 = (A X_1 + B u_1) dt + B_1 u_1  d W_1 ,  \label{inverse X_1}
\end{align}
where $X_1(0)$ is given.
The agent chooses the control $u_1$ to minimize the cost
\begin{align}
J_1(u_1) = \mathbb{E} \bigg[ \int_0^T ([X_1]_Q^2 + [u_1]_R^2 )dt + [X_1(T) ]_{Q_f}^2 \bigg] . \notag
\end{align}
If \eqref{ODELam1} admits a solution $\Lambda_1$, then the optimal control is
\begin{align}
u_1(t) = ( R + B_1^T \Lambda_1(t) B_1  )^{-1} B^T \Lambda_1(t) X_1(t)
. \notag
\end{align}

With $\Lambda_1$ obtained by solving \eqref{ODELam1}, we consider another
single-agent optimal control problem with state dynamics
\begin{align}
d X_2 = [ (A + G) X_2 + B u_2  ] dt + B_0 u_2  d W_0 ,
\label{inverse X_2}
\end{align}
and the agent chooses  $u_2$ to minimize the cost
\begin{align}
J_2(u_2) = \mathbb{E} \bigg[ \int_0^T( [X_2]_{Q_3}^2
+ [u_2]_{R + B_1^T \Lambda_1 B_1}^2) dt
+ [X_2(T)]_{Q_{3f}}^2 \bigg] . \notag
\end{align}
If \eqref{ODELam3} admits a solution $\Lambda_3$, then the optimal control $u_2$ is given by
\begin{align}
u_2(t) = ( R + B_1^T \Lambda_1(t) B_1 + B_0^T \Lambda_3(t) B_0 )^{-1}
B^T \Lambda_3(t) X_2(t)   . \notag
\end{align}

\section{Closed-loop dynamics and mean field limit}
\label{sec: cl dynamics}

We introduce the following assumptions:
\begin{assumption}
\label{assm: sol Lambda Lambda_1}
The ODE system \eqref{ODELam1}--\eqref{ODELam2} has a solution $(\Lambda_1 , \Lambda_2)$ on $[0, T]$.
\end{assumption}

Let $\{X_i(0), 1\le i\le N\}$  be the initial states  of the $N$ agents.
Denote the covariance matrix $\Sigma^{i}_0:=\mathrm{Cov}(X_i(0), X_i(0))$, $1\leq i \leq N$.

\begin{assumption} \label{assumption:mean} The initial states $\{X_i(0), i\ge 0\}$ are independent.
There exist a mean $\mu_0\in \mathbb{R}^n$ and a constant $C_\Sigma$, both independent of $N$, such that
$\mathbb{E}X_i(0) = \mu_0$ and $| \Sigma_0^{i} | \leq C_\Sigma$
for all $i$.
\end{assumption}

If \eqref{ODELam12N} has a solution $(\Lambda_1^N, \Lambda_2^N)$ for a finite $N$, by Lemma \ref{lemma:P2Lam} (ii), we determine $\mathbf{P}$   in
\eqref{U 2} by \eqref{P submatrices} with $  \Pi_1^N=  \Lambda_1^N $ and $ \Pi_2^N =  \Lambda_2^N/N $. Then we obtain the optimal  control
$U^o=(u_1^T, \cdots, u_N^T)^T$, where
\begin{align}
u_i = - \Theta^N X_i - \Theta_1^N X^{(N)} -  \Theta_2^N , \quad
 1\leq i \leq N ,
\label{centralized u_i}
\end{align}
and
\begin{align}
& \Theta^N =  ( H^N - E^N) B^T ( \Lambda_1^N - {\Lambda_2^N}/{N} ) , \nonumber\\
& \Theta_1^N =  N E^N B^T \Lambda_1^N + (H^N + (N-2) E^N) B^T \Lambda_2^N   , \notag \\
 & \Theta_2^N =
( H^N + (N-1) E^N )
  [ B^T S^N +   B_1^T \Lambda_1^N D  +  B_0^T(\Lambda_1^N + (1-1/N) \Lambda_2^N) D_0 ] . \notag
\end{align}
The control $U^o=(u_1^T, \cdots, u_N^T)^T$ given by \eqref{centralized u_i} is called \emph{centralized}, as each agent $\mathcal{A}_i$ needs the state information of other agents and also the population size $N$.

Denote the closed-loop dynamics of $X_i$ and $X^{(N)}$ by
\begin{align}
\label{CL X_i}
d X_i = & [ ( A - B \Theta^N )X_i + (G-B \Theta_1^N) X^{(N)} - B \Theta_2^N ] dt
 \\
 &  + [ D - B_1 ( \Theta^N X_i + \Theta_1^N X^{(N)} + \Theta_2^N ) ] d W_i   \notag \\
&  + [ D_0 - B_0 ((\Theta^N + \Theta_1^N) X^{(N)} + \Theta_2^N) ] d W_0 ,
\quad 1\leq i \leq N ,  \notag
\end{align}
\begin{align}
\label{CL X^N}
d X^{(N)} = & [  ( A + G - B ( \Theta^N + \Theta_1^N ) ) X^{(N)} - B \Theta_2^N ] dt  \\
 &  + (1/N) \sum_{i=1}^N [ D - B_1 ( \Theta^N X_i + \Theta_1^N X^{(N)} + \Theta_2^N ) ] d W_i   \notag \\
 &  + [  D_0 - B_0 ( \Theta^N + \Theta_1^N) X^{(N)}  - B_0 \Theta_2^N    ] d W_0 .  \notag
\end{align}

Let $(\Lambda_1(t), \Lambda_2(t))$ be given by Assumption \ref{assm: sol Lambda Lambda_1} and denote matrix-valued functions on $[0,T]$:
\begin{align}
& H   = (  \mathcal{R}_1(\Lambda_1) )^{-1},
\quad
 H_1 = (  \mathcal{R}_2(\Lambda_1,  \Lambda_2) )^{-1}, \quad
\widehat{E} = H_1 - H, \notag \\
&  \Theta = H B^T \Lambda_1 , \quad
\Theta_1  = \widehat{E}B^T( \Lambda_1 + \Lambda_2 ) + H B^T \Lambda_2 ,
 \nonumber \\
& \Theta_2  = H_1 \left[ B^T S +  B_1^T \Lambda_1 D  +  B_0^T(\Lambda_1
+ \Lambda_2 ) D_0 \right] .
   \notag
\end{align}
Denote the mean field limit of \eqref{CL X^N}:
\begin{align}
d \overline X  =\ & [ (A+G - B(\Theta + \Theta_1)) \overline{X} - B \Theta_2 ] dt  \label{CL bar X} \\
& + [ D_0 - B_0 \Theta_2 - B_0(\Theta + \Theta_1 ) \overline{X}   ] d W_0
,
\notag
\end{align}
with the initial condition $\overline{X}(0)=\mu_0\in \mathbb{R}^n$.
We proceed to check the approximation error between $X^{(N)}$ and $\overline X$.

\begin{lemma}
\label{lm: Theta^N-Theta}
Under Assumption~\ref{assm: sol Lambda Lambda_1}, we have
$\sup_{t\in[0, T]} ( |\Theta^N - \Theta| + |\Theta_1^N - \Theta_1|
+ |\Theta_2^N - \Theta_2| ) = O(1/N)$.
\end{lemma}
\begin{proof}
It follows  from Corollary~\ref{cor: sol Lambda_1^N Lambda_2^N}.
\end{proof}

\begin{lemma}
\label{lm: ub X_i}
Under Assumptions~\ref{assm: sol Lambda Lambda_1} and \ref{assumption:mean},   there exist $C>0$ and $N_0 >0$ such that
$\sup_{i \geq N_0 , 0 \leq t\leq T} \mathbb{E} |X_i(t)|^2 \leq C$, where $X_i(t)$ is given by \eqref{CL X_i}.
\end{lemma}

\begin{proof}
By Assumption \ref{assumption:mean},
we have  $\sup_{i \geq 1 , 0 \leq t\leq T} \mathbb{E} |X_i(0)|^2 \leq C_0$ for some fixed constant $C_0$.
Applying It\^{o}'s formula to $|X_i|^2$ gives
\begin{align}
 \mathbb{E}|X_i (t)|^2 = &\ \mathbb{E} |X_i(0)|^2  + \mathbb{E} \int_0^t
2
\langle X_i,  ( A - B \Theta^N )X_i + (G-B \Theta_1^N) X^{(N)} - B \Theta_2^N  \rangle d s    \notag \\
 & +  \mathbb{E} \int_0^t  | D - B_1 ( \Theta^N X_i + \Theta_1^N X^{(N)} + \Theta_2^N ) |^2  ds  \notag \\
 & + \mathbb{E} \int_0^t  | D_0 - B_0 ((\Theta^N + \Theta_1^N) X^{(N)} + \Theta_2^N) |^2  ds  . \notag
\end{align}
By Lemma~\ref{lm: Theta^N-Theta}, $(\Theta^N(t), \Theta_1^N(t), \Theta_2^N(t))$ is uniformly bounded on $[0,T]$ for all large $N$. So
there exist $N_0>0$ and constants $C_1$, $C_2$ and $C_3$ such that $N\geq N_0$ implies
 \begin{align}
\mathbb{E} |X_i(t)|^2 \leq \mathbb{E}|X_i(0)|^2 + C_1 + C_2 \int_0^t \mathbb{E}|X_i(s)|^2 ds
+ C_3 \int_0^t \mathbb{E}|X^{(N)}(s)|^2 ds  \notag
\end{align}
for all $t\in [0,T]$ and all $1\le i\le N$.
Denote $\alpha^N_t = \max_{1\leq k \leq N} \mathbb{E} |X_k(t)|^2$. Note
that
$\mathbb{E} |X^{(N)}(t)|^2
 \le  (1/N) \sum_{i=1}^N \mathbb{E} |X_i(t)|^2 $.
It then follows that for any $1\le i\le N$,
\begin{align}
\mathbb{E}|X_i(t)|^2 \leq \max_{1\leq k \leq N} \mathbb{E}|X_k(0)|^2 + C_1 + C_2 \int_0^t \alpha^N_s ds + C_3 \int_0^t \alpha^N_s ds , \notag
\end{align}
and  therefore
\begin{align}
\alpha^N_t \leq  C_0 + C_1 + (C_2 + C_3)
\int_0^t \alpha^N_s ds . \notag
\end{align}
Gr\"{o}nwall's lemma implies that
$\alpha^N_t \leq C$ for all $t\in [0, T]$ and all $N\geq N_0$, where
the constant $C$ depends only on $C_0$, $C_1$, $C_2$, $C_3$ and $T$.
\end{proof}

\begin{proposition}
\label{prop: X^N-bar X}

Under Assumptions~\ref{assm: sol Lambda Lambda_1} and  \ref{assumption:mean},
for \eqref{CL X^N}--\eqref{CL bar X}
it holds  that
\begin{align}
\sup_{t\in [0, T]} \mathbb{E} | X^{(N)}(t) - \overline{X}(t) |^2 =
 O(  1/N ) .
\label{X^N-bar X}
\end{align}
\end{proposition}

\begin{proof}
Taking the difference between \eqref{CL X^N} and \eqref{CL bar X} gives
\begin{align}
d ( X^{(N)} - \overline{X} ) = & [ (A+G - B(\Theta + \Theta_1)) (X^{(N)} - \overline{X} )\\
&-B(\Theta^N + \Theta_1^N - \Theta - \Theta_1 ) X^{(N)}
  - B(\Theta_2^N - \Theta_2) ] dt \notag \\
& - [  B_0(\Theta + \Theta_1 ) (X^{(N)} - \overline{X} ) +B_0(\Theta^N + \Theta_1^N - \Theta - \Theta_1 )X^{(N)} \nonumber \\
& + B_0(\Theta_2^N - \Theta_2)] d W_0 \notag \\
&  + \frac{1}{N} \sum_{i=1}^N [ D - B_1 ( \Theta^N X_i  + \Theta_1^N \overline{X} + \Theta_2^N )  ]  d W_i . \notag
\end{align}
We apply It\^{o}'s formula to $|X^{(N)} - \overline{X}|^2$ to get
\begin{align}
 &\mathbb{E} |X^{(N)}(t) - \overline{X}(t) |^2 \nonumber\\
 = \ & \mathbb{E} |X^{(N)}(0) - \overline{X}(0)|^2 +  \frac{1}{N^2} \sum_{i=1}^N \int_0^t
 \mathbb{E} | D - B_1 ( \Theta^N X_i  + \Theta_1^N \overline{X} + \Theta_2^N )  |^2  d s  \notag \\
& + 2 \int_0^t  \mathbb{E} \langle  X^{(N)} - \overline{X} ,
(A + G - B (\Theta + \Theta_1 )  )  ( X^{(N)} - \overline{X} )  \rangle  ds  \notag \\
& + 2 \int_0^t \mathbb{E} \langle X^{(N)} - \overline{X} , - B (\Theta^N + \Theta_1^N - \Theta - \Theta_1 ) X^{(N)} - B (\Theta_2^N - \Theta_2) \rangle ds \notag \\
& + \int_0^t \mathbb{E} |  B_0(\Theta + \Theta_1)(X^{(N)}-\overline{X})
 + B_0( \Theta^N + \Theta_1^N - \Theta - \Theta_1 ) X^{(N)}\nonumber\\
 &\qquad\quad + B_0 (\Theta_2^N - \Theta_2 ) |^2 ds    . \notag
\end{align}
By Cauchy-Schwarz inequality and Lemmas~\ref{lm: Theta^N-Theta}
and \ref{lm: ub X_i}, it holds that for all sufficiently large $N$,
\begin{align}
 \mathbb{E}|X^{(N)}(t) - \overline{X}(t) |^2
\leq & \frac{C}{N} + \frac{1}{N^2} \sum_{i=1}^N \int_0^t C ds  \notag \\
 &\hskip -2cm + C\int_0^t  \mathbb{E}|X^{(N)}(s) - \overline{X}(s) |^2
+ |\Theta^N(s) - \Theta(s)|^2 + \sum_{k=1}^2 |\Theta_k^N(s) - \Theta_k(s)|^2 ds
\notag \\
\leq &  \frac{C_1}{N}
 + C\int_0^t \mathbb{E}|X^{(N)}(s) - \overline{X}(s)|^2 ds . \notag
\end{align}
The estimate \eqref{X^N-bar X}  follows from  Gr\"{o}nwall's lemma.
\end{proof}

Below we give a closed-form expression of the individual cost under $U^o$
by assuming that $\{X_i(0): 1\leq i \leq N\}$ are independent random variables with equal mean and covariance
\begin{align}
\mathbb{E}X_i(0) = \mu_0 , \quad \mathrm{Cov}(X_i(0), X_i(0)) = \Sigma_0 , \quad \forall  i \geq 1 .
\label{meancovX0}
\end{align}
Then the individual cost of a single agent $\mathcal{A}_i$ is
\begin{align}
   J_i(U^o)  = &(1/N ) \mathbb{E} {V}(0, X(0))
   \label{jiuoN} \\
= & ({1}/{N}) \mathbb{E} [ X^T(0)   \mathbf{P}(0)  X(0)
  +  \mathbf{S}^T(0) X(0) + {\bf r}(0) ]   \notag \\
 = & \mathbb{E} [ X_1^T(0) \Lambda_1^N(0) X_1(0) + (1-1/N) X_1^T(0) \Lambda_2^N(0) X_2(0) \notag \\
& + 2 S^{NT}(0) X_1(0) +  r^N(0) ] \notag \\
=& \mathrm{Tr} [ \Lambda_1^N(0) \Sigma_0 ]
+\mu_0^T( \Lambda_1^N(0) + (1-1/N) \Lambda_2^N(0) ) \mu_0
 + 2 S^{NT}(0) \mu_0 + r^N(0)  . \notag
\end{align}

\section{Decentralized  control}
\label{sec: performance analysis}

It is desirable  to find a decentralized control such that each agent only needs to know its own state and some low-dimensional auxiliary state.
Based on  the mean field limit dynamics \eqref{CL bar X},
we consider a decentralized  control law ${U}^d = (\check{u}_1^T, ..., \check{u}_N^T)^T$, where the individual control law
is \begin{align}
& \check u_i    =  - \Theta X_i   - \Theta_1 \overline X - \Theta_2 ,
\quad 1\leq i \leq N .
\label{check u_i}
\end{align}
We may view $\check u_i$ as the mean field limit of \eqref{centralized u_i}.
Note that without the common noise, $\overline X(t)$ becomes a deterministic function and may be computed off-line.

For  the decentralized control
 applied to the $N$-agent model, we have the  main result.
\begin{theorem}
\label{thm: performance analysis}
Under Assumptions~\ref{assm: sol Lambda Lambda_1} and \ref{assumption:mean},
let $U^o=(u_1^T, \cdots, u_N^T)^T$ be the centralized optimal control given by \eqref{centralized u_i},
and ${U}^d=(\check{u}_1^T, \cdots, \check{u}_N^T)^T$  the decentralized
control given by \eqref{check u_i}.
Then $0\le  J_{\rm soc}^{(N)}({U}^d) - J_{\rm soc}^{(N)}(U^o)  = O(1)$.
\end{theorem}
Theorem~\ref{thm: performance analysis} shows that if the decentralized  control \eqref{check u_i} is applied,
 the  optimality gap $J_{\rm soc}^{(N)}({U}^d) - J_{\rm soc}^{(N)}(U^o)$ is bounded, independent of the population size $N$. It is easy to show $J_{\rm soc}^{(N)}(U^o)=O(N)$. This  bounded optimality gap means
an  optimality loss of $O(1/N)$ per agent.
To prove the theorem, we will
find $J_{\rm soc}^{(N)}({U}^d)$ in an explicit form.

\subsection{Social cost under decentralized control}

Let $(\Lambda_1, \Lambda_2)$ be a solution of
\eqref{ODELam1}--\eqref{ODELam2} under Assumption~\ref{assm: sol Lambda Lambda_1}.

The $N$-agent system \eqref{X_i} under the set of decentralized individual control laws~\eqref{check u_i} has the following closed-loop dynamics
\begin{align}
 d \check{X}_i = & [ ( A - B\Theta) \check{X}_i + G \check{X}^{(N)} - B\Theta_1 \overline{X} - B\Theta_2 ] dt \notag \\
& + [D - B_1(\Theta \check{X}_i + \Theta_1 \overline{X} + \Theta_2 )] d W_i \notag \\
& + [D_0 - B_0(\Theta \check{X}^{(N)} + \Theta_1 \overline{X} + \Theta_2 )] d W_0 , \quad 1\leq i \leq N , \label{xwithud}
\end{align}
where $\check X_i(0)=X_i(0)$ and $\overline{X}$ satisfies \eqref{CL bar
X}.

Let $X(0) = (X_1^T(0), \cdots, X_N^T(0))^T$.
Note that given $(X(0), \overline X(0))$, the processes $X_1, \cdots, X_N$, and $\overline X$ have been determined on $[0,T]$. In order to evaluate $J^{(N)}_{\rm soc}( U^d)$, we consider a family of SDEs \eqref{CL bar X} and \eqref{xwithud} by assigning different initial conditions. We take the initial time $t$ and assign the initial condition $X(t)=\mathbf{x}$, $\overline X(t)=\bar x$.

By extending \eqref{J soc 2} to different initial conditions,
we define the social cost
\begin{align}
\check{V}(t, \mathbf{x}, \bar{x})=
&\sum_{i=1}^N \bigg\{ \mathbb{E}  \int_t^T \Big[ [\check X_i(s) - \Gamma \check X^{(N)}(s)]_Q^2 + [\check u_i(s)]_{R}^2  \Big]   dt \label{VEtx}  \\
& \quad  +\mathbb{E}  [\check X_i(T) - \Gamma_f \check X^{(N)}(T)]_{Q_f}^2
\bigg\} \nonumber
\end{align}
 with initial condition $(X(t), \overline{X}(t)) = (\mathbf{x}, \bar{x})$ under the decentralized control
${U}^d(s) = (\check{u}_1^T(s), \cdots, \check{u}_N^T(s))^T$,
$s\in [t, T]$.
Recalling \eqref{check u_i}, below we write the state feedback control ${U}^d$ as
\begin{align}
& {U}^d(t,\mathbf{x}, \bar x) = (\check{u}_1^T, \cdots, \check{u}_N^T)^T
=  - \widehat{\Theta} \mathbf{x} - {\widehat\Theta}_1 \bar x - {\widehat\Theta}_2 , \label{check U} \\
 & \widehat{\Theta}  = I_N \otimes  \Theta , \quad \widehat{\Theta}_1 =
\mathbf{1}_{N\times 1} \otimes \Theta_1, \quad \widehat{\Theta}_2
 = \mathbf{1}_{N\times 1} \otimes \Theta_2 .   \notag
\end{align}
By the Feynman--Kac formula \cite[Sec. 1.3, 3.5]{P2009},   the equation  for $\check{V}(t, \mathbf{x}, \bar{x})$ is given as
\begin{align}
\begin{cases}
-\frac{\partial \check{V}}{\partial t}
= {U}^{dT} (\mathbf{R} + \mathcal{M}_2(\tfrac{\partial^2 \check{V}}{\partial \mathbf{x}^2})) {U}^d + ( \tfrac{\partial^T \check{V} }{\partial \mathbf{x}} \widehat{\mathbf{B}} + \mathcal{M}_1(\tfrac{\partial^2 \check{V} }{\partial \mathbf{x}^2}) ) {U}^d \\
\hspace{1.1cm}
+ \tfrac{\partial^T \check{V} }{\partial \mathbf{x}}
 \mathbf{A} \mathbf{x} + \tfrac{\partial^T \check{V}}{\partial \bar{x}} (
Z_1 \bar{x} - B \Theta_2 )
   + \mathbf{x}^T \mathbf{Q} \mathbf{x} + \mathcal{M}_0(\tfrac{\partial^2
\check{V}}{\partial \mathbf{x}^2})   \\
\hspace{1.1cm}
+ (1/2) ( Z_0 \bar{x} - B_0 \Theta_2 + D_0 )^T \tfrac{\partial^2 \check{V}}{\partial \bar{x}^2 } ( Z_0 \bar{x} - B_0 \Theta_2 + D_0 )
 \\
\hspace{1.1cm}
+  (\mathbf{B}_0 \widehat{\mathbf{I}} U^d + \mathbf{D}_0)^T \tfrac{\partial^2 \check{V}}{\partial \mathbf{x} \partial \bar{x}} (Z_0 \bar{x} - B_0 \Theta_2 + D_0 ) , \\
 \check{V}(T, \mathbf{x} , \bar{x})  =  \mathbf{x}^T \mathbf{Q}_f \mathbf{x}  ,
\end{cases}
\label{HJB check V}
\end{align}
where we denote
\begin{align}
 & Z_0
= - B_0 (\Theta + \Theta_1 ) ,   \quad
 Z_1
= A + G - B (\Theta + \Theta_1).
\notag
\end{align}
Thus the right hand side of \eqref{VEtx} is just a probabilistic representation of the solution of \eqref{HJB check V} which will be determined below.

Suppose $\check{V}$ takes the following form
\begin{align}
\check{V}(t, \mathbf{x}, \bar{x}) =& \mathbf{x}^T \check{\mathbf{P}}_1(t) \mathbf{x} + \bar{x}^T
 \check{\mathbf{P}}_2(t) \bar{x}
 + \mathbf{x}^T \check{\mathbf{P}}_{12}(t) \bar{x} + \bar{x}^T \check{\mathbf{P}}_{12}^T(t) \mathbf{x}
\label{check V ansatz} \\
& + 2 \mathbf{x}^T \check{\mathbf{S}}_1(t) + 2 \bar{x}^T \check{\mathbf{S}}_2(t) + \check{\mathbf{r}}(t) .  \notag
\end{align}
By substituting \eqref{check V ansatz} into \eqref{HJB check V} and combining like terms, we obtain for $\check{\mathbf{P}}_1$,
$\check{\mathbf{P}}_{12}$, $\check{\mathbf{P}}_2$, $\check{\mathbf{S}}_1$, $\check{\mathbf{S}}_2$, and $\check{\mathbf{r}}$ the ODEs:
\begin{align}
&\begin{cases}
  - \tfrac{d}{dt} \check{\mathbf{P}}_1  =  \widehat{\Theta}^T
(\mathbf{R} + \mathcal{M}_2(2\check{\mathbf{P}}_1 ) ) \widehat\Theta  +  \check{\mathbf{P}}_1 (\mathbf{A} - \widehat{\mathbf{B}} \widehat\Theta )
        \\
 \qquad \qquad     +   (\mathbf{A} - \widehat{\mathbf{B}} \widehat\Theta )^T \check{\mathbf{P}}_1  + \mathbf{Q} , \\
  \check{\mathbf{P}}_1(T)  =  \mathbf{Q}_f ,
\end{cases}
\label{ODE check P_1} \\
&\begin{cases}
-  \tfrac{d}{dt} \check{\mathbf{P}}_{12} =   \widehat\Theta^T (\mathbf{R}  + \mathcal{M}_2(2\check{\mathbf{P}}_1) ) \widehat\Theta_1 -  \check{\mathbf{P}}_1 \widehat{\mathbf{B}} \widehat\Theta_1
 + ( \mathbf{A}^T -  \widehat\Theta^T \widehat{\mathbf{B}}^T ) \check{\mathbf{P}}_{12}  \\
 \hspace{1.6cm} +  \check{\mathbf{P}}_{12} Z_1
 - \widehat\Theta^T \widehat{\mathbf{I}}^T \mathbf{B}_0^T \check{\mathbf{P}}_{12} Z_0 ,  \\
\check{\mathbf{P}}_{12}(T) = 0  ,
\end{cases} \label{ODE check P_{12}}
\end{align}
\begin{align}
& \begin{cases}
-\tfrac{d}{dt} \check{\mathbf{P}}_2 =  \widehat\Theta_1^T (\mathbf{R} +
\mathcal{M}_2(2\check{\mathbf{P}}_1) ) \widehat\Theta_1
-  \check{\mathbf{P}}_{12}^T \widehat{\mathbf{B}} \widehat\Theta_1
- \widehat\Theta_1^T \widehat{\mathbf{B}}^T \check{\mathbf{P}}_{12}
  \\
\qquad\qquad
- Z_0^T \check{\mathbf{P}}_{12}^T \mathbf{B}_0 \widehat{\mathbf{I}} \widehat\Theta_1 - \widehat\Theta_1^T \widehat{\mathbf{I}}^T \mathbf{B}_0^T \check{\mathbf{P}}_{12} Z_0
+  \check{\mathbf{P}}_2 Z_1 + Z_1^T \check{\mathbf{P}}_2     \\
\qquad\qquad
+ Z_0^T \check{\mathbf{P}}_2 Z_0 , \\
\check{\mathbf{P}}_2(T) = 0  ,
\end{cases}  \label{ODE check P_2} \\
& \begin{cases}
- \frac{d}{dt} \check{\mathbf{S}}_1 = \widehat\Theta^T (\mathbf{R} + \mathcal{M}_2(2\check{\mathbf{P}}_1) ) \widehat{\Theta}_2
-  \widehat\Theta^T ( \widehat{\mathbf{B}}^T \check{\mathbf{S}}_1 + \mathcal{M}_1^T( \check{\mathbf{P}}_1) ) \\
 \qquad\qquad
+  \mathbf{A}^T \check{\mathbf{S}}_1
 - \check{\mathbf{P}}_{12} B \Theta_2
-  \widehat\Theta^T \widehat{\mathbf{I}}^T \mathbf{B}_0^T  \check{\mathbf{P}}_{12} (D_0 - B_0\Theta_2)   \\
\qquad\qquad
  - \check{\mathbf{P}}_1 \widehat{\mathbf{B}} \widehat{\Theta}_2 , \\
\check{\mathbf{S}}_1(T) = 0 ,
\end{cases} \label{ODE check S_1} \\
&\begin{cases}
- \frac{d}{dt} \check{\mathbf{S}}_2 = \widehat{\Theta}_1^T (\mathbf{R} + \mathcal{M}_2(2\check{\mathbf{P}}_1) ) \widehat{\Theta}_2
- \widehat{\Theta}_1^T (\widehat{\mathbf{B}}^T \check{\mathbf{S}}_1 + \mathcal{M}_1^T( \check{\mathbf{P}}_1))  \\
\qquad\qquad - \check{\mathbf{P}}_{12}^T \widehat{\mathbf{B}}\widehat{\Theta}_2
- \check{\mathbf{P}}_2 B \Theta_2 + Z_1^T \check{\mathbf{S}}_2
+ Z_0^T \check{\mathbf{P}}_2 (D_0 - B_0 \Theta_2) \\
 \qquad\qquad + Z_0^T \check{\mathbf{P}}_{12}^{T}
( \mathbf{D}_0 - \mathbf{B}_0 \widehat{\mathbf{I}} \widehat\Theta_2)
 - \widehat\Theta_1^T \widehat{\mathbf{I}}^T \mathbf{B}_0^T \check{\mathbf{P}}_{12} (D-B_0 \Theta_2) ,  \\
\check{\mathbf{S}}_2(T) = 0 ,
\end{cases} \label{ODE check S_2} \\
& \begin{cases}
- \frac{d}{dt}\check{\mathbf{r}} = \widehat\Theta_2^T (\mathbf{R} + \mathcal{M}_2(2\check{\mathbf{P}}_1) ) \widehat\Theta_2 - 2 (\check{\mathbf{S}}_1^T \widehat{\mathbf{B}}
+ \mathcal{M}_1(\check{\mathbf{P}}_1) ) \widehat\Theta_2
- 2\check{\mathbf{S}}_2^{ T} B\Theta_2  \\
\hspace{1.3cm}
+ (D_0 - B_0 \Theta_2 )^T ( \check{\mathbf{P}}_2 + 2 \widehat{\mathbf{I}}
\check{\mathbf{P}}_{12}  )(D_0 - B_0 \Theta_2)
+ \mathcal{M}_0(2\check{\mathbf{P}}_1) ,    \\
\check{\mathbf{r}}(T) = 0 .
\end{cases}  \label{ODE check r}
\end{align}

The above is a  system of six linear ODEs and has a unique solution on $[0,T]$.
By the following Lemmas \ref{lm: check P_1 submatrices}
and \ref{lm: check P_{12} submatrices},  we further obtain the low-dimensional ODE systems corresponding to the high-dimensional systems \eqref{ODE check P_1} and \eqref{ODE check P_{12}}.
The proof of Lemma~\ref{lm: check P_1 submatrices}
is similar to that of Lemma~\ref{lm: P submatrices},
and the proofs of Lemmas
\ref{lm: check P_{12} submatrices} and
\ref{lm: check S_1 submatrices}
are similar to that of Lemma~\ref{lm: S submatrices}; we omit the details
here.
\begin{lemma}
\label{lm: check P_1 submatrices}
For \eqref{ODE check P_1}, the solution $\check{\mathbf{P}}_1$ on $[0, T]$ has the representation
\begin{align}
\check{\mathbf{P}}_1 = \begin{bmatrix} \check{\Pi}_1^N & \check{\Pi}_2^N &  \cdots & \check{\Pi}_2^N \\
\check{\Pi}_2^N & \check{\Pi}_1^N & \cdots & \check{\Pi}_2^N \\
\vdots & \vdots & \ddots & \vdots \\
\check{\Pi}_2^N & \check{\Pi}_2^N & \cdots & \check{\Pi}_1^N  \end{bmatrix} ,
\quad
\check\Pi_1^N(t), \  \check\Pi_2^N(t) \in \mathcal{S}^n .
 \label{check P_1 submatrices}
\end{align}
\end{lemma}

\begin{lemma}
\label{lm: check P_{12} submatrices}
The matrix $\check{\mathbf{P}}_{12}(t)\in \mathbb{R}^{N n \times n}$ takes the form
\begin{align}
\check{\mathbf{P}}_{12}(t) = ( \check\Pi_{12}^{N T}(t), \cdots, \check\Pi_{12}^{N T} (t) )^T ,
\quad  \check\Pi_{12}^N(t) \in \mathbb{R}^{n \times n} .
\label{check P_{12} submatrices}
\end{align}
\end{lemma}

\begin{lemma}
\label{lm: check S_1 submatrices}
The matrix $\check{\mathbf{S}}_{1}(t) \in \mathbb{R}^{Nn \times 1}$ takes
the form
\begin{align}
\check{\mathbf{S}}_1(t) = (\check{S}_1^{N T}(t), \cdots, \check{S}_1^{N
T}(t) )^T,
\quad \check{S}_1^N (t)\in \mathbb{R}^{n\times 1}  .
\label{check S_1 submatrices}
\end{align}
\end{lemma}

Following the rescaling method in Section \ref{sec:sub:as}, we define
\begin{align}
& \check\Lambda_1^N =\check\Pi_1^N ,\hspace{.15cm}
\check\Lambda_2^N=N\check\Pi_2^N  ,\hspace{.15cm}
\check\Lambda_{12}^N= \check\Pi_{12}^N,\hspace{.15cm}
  \check\Lambda_{22}^N =\check{\mathbf{P}}_2/N, \hspace{.15cm}
    \check{S}_2^N= \check{\mathbf{S}}_2/N,  \hspace{.15cm}
  \check{r}^N=\check{\mathbf{r}}/N. \notag
  \end{align}
We give some intuition behind the scaling used to define $\check \Lambda_{22}^N$.
Take $\mathbf{x}=0$ and a large  $|\bar x|>  0$ at $t=0$; the resulting control input will generate  processes $\{X_i,1\le i\le N \}$ each containing a constituent component of roughly the magnitude of $\bar x$.
Then the social cost will contain a component of magnitude $O(N |\bar x|^2)$.
This suggests $\check{\mathbf{P}}_2$ increases nearly linearly with $N$.
We substitute \eqref{check P_1 submatrices}, \eqref{check P_{12} submatrices}
and \eqref{check S_1 submatrices} into
\eqref{ODE check P_1}, \eqref{ODE check P_{12}} and \eqref{ODE check S_1},
with
$\check\Pi_1^N = \check\Lambda_1^N$, $\check\Pi_2^N = \check\Lambda_2^N/N$, $\check\Pi_{12}^N = \check\Lambda_{12}^N$.
We further rewrite \eqref{ODE check P_2}, \eqref{ODE check S_2} and \eqref{ODE check r} using the new variables.
After the change of variables, we derive
\begin{align}
&\begin{cases}
-\frac{d}{dt} \check{\Lambda}_1^N
=  \Theta^T \mathcal{R}_1(\check{\Lambda}_1^N ) \Theta
+ \check{\Lambda}_1^N (A - B \Theta )   \\
 \hspace{1.6cm} + (A - B \Theta )^T \check{\Lambda}_1^N + Q + \check{g}_1(N, \check{\Lambda}_1^N, \check{\Lambda}_2^N ) ,   \\
\check{\Lambda}_1^N(T)  =  Q_f
+Q^\Gamma_f/N ,
\end{cases}
\label{ODEcheckLam1N} \\
&\begin{cases}
-\frac{d}{dt} \check{\Lambda}_2^N
=  \Theta^T B_0^T ( \check{\Lambda}_1^N + \check{\Lambda}_2^N ) B_0 \Theta
 + \check{\Lambda}_1^N G  + G^T \check{\Lambda}_1^N    \\
\hspace{1.6cm} + \check{\Lambda}_2^N (A + G - B \Theta )
  + (A + G - B \Theta )^T  \check{\Lambda}_2^N  \\
\hspace{1.6cm}
+Q^\Gamma + \check{g}_2(N,  \check{\Lambda}_2^N ) ,
\\
 \check{\Lambda}_2^N(T) = Q^\Gamma_f ,
\end{cases}
\label{ODEcheckLam2N} \\
& \begin{cases}
- \frac{d}{dt} \check\Lambda_{12}^N =
\Theta^T  \mathcal{R}_2(\check\Lambda_1^N, \check\Lambda_2^N)   \Theta_1
 + \Theta^T B_0^T \check\Lambda_{12}^N B_0  ( \Theta + \Theta_1 )   \\
  \hspace{1.6cm}
- ( \check\Lambda_1^N + \check\Lambda_2^N ) B\Theta_1
 + [ A + G - B \Theta  ]^T \check\Lambda_{12}^N
   \\
 \hspace{1.6cm}
 + \check\Lambda_{12}^N [A + G - B ( \Theta_1 + \Theta ) ]
 + \check{g}_{12}(N,  \check\Lambda_2^N ) , \\
\check\Lambda_{12}^N(T) = 0 ,
\end{cases} \label{ODEcheckLam12N} \\
& \begin{cases}
- \frac{d}{dt} \check\Lambda_{22}^N =  \Theta_1^T
 \mathcal{R}_2( \check\Lambda_1^N ,  \check\Lambda_2^N )   \Theta_1
 -  \check\Lambda_{12}^{N T} B \Theta_1
 -  \Theta_1^T B^T \check\Lambda_{12}^N  \\
 \hspace{1.6cm}
 +  \check\Lambda_{22}^N Z_1 + Z_1^T \check\Lambda_{22}^N
  -  Z_0^T \check\Lambda_{12}^{N T} B_0 \Theta_1
 -  \Theta_1^T B_0^T \check\Lambda_{12}^N Z_0
      \\
 \hspace{1.6cm}  +  Z_0^T \check\Lambda_{22}^N Z_0 +  \check{g}_{22}(N,  \check\Lambda_2^N) ,    \\
\check\Lambda_{22}^N(T) = 0 ,
\end{cases}     \label{ODEcheckLam22N}  \\
& \begin{cases}
-\frac{d}{dt} \check{S}_1^N = \Theta^T  \mathcal{R}_2( \check\Lambda_1^N ,  \check\Lambda_2^N )  \Theta_2
- ( \check\Lambda_1^N +  \check\Lambda_2^N + \check\Lambda_{12}^N ) B \Theta_2  \\
\hspace{1.6cm}
 - \Theta^T [B^T \check{S}_1^N + B_1^T \check\Lambda_1^N D + B_0^T(\check\Lambda_1^N + \check\Lambda_2^N )D_0]   \\
\hspace{1.6cm}
- \Theta^T B_0^T \check\Lambda_{12}^N (D_0 - B_0 \Theta_2 )
 + (A^T + G^T )\check{S}_1^N  \\
\hspace{1.6cm} + \check{g}_{01}(N,  \check\Lambda_2^N)  , \\
\check{S}_1^N(T) = 0 ,
\end{cases} \label{ODEcheckS1N}
\end{align}
\begin{align}
&\begin{cases}
- \frac{d}{dt} \check{S}_2^N =   \Theta_1^T \mathcal{R}_2(\check\Lambda_1^N, \check\Lambda_2^N) \Theta_2 + Z_1^T \check{S}_2^N   \\
\hspace{1.6cm}
-  \Theta_1^T [ B^T \check{S}_1^N + B_1^T \check\Lambda_1^N D
+ B_0^T (\check\Lambda_1^N + \check\Lambda_2^N ) D_0 ]
 \\
\hspace{1.6cm}
+ ( Z_0^T \check\Lambda_{22}^N +  Z_0^T \check\Lambda_{12}^{N T}
 -  \Theta_1^T B_0^T  \check\Lambda_{12}^N )  (D_0 - B_0  \Theta_2)
\\
\hspace{1.6cm}
 - (  \check\Lambda_{12}^{N T} + \check\Lambda_{22}^N ) B\Theta_2
+ \check{g}_{02}(N,  \check\Lambda_2^N) ,  \\
 \check{S}_2^N(T) = 0 ,
\end{cases} \label{ODEcheckS2N}  \\
&\begin{cases}
- \frac{d}{dt}  \check{r}^N  = \Theta_2^T  \mathcal{R}_2(\check\Lambda_1^N, \check\Lambda_2^N)  \Theta_2
 + D^T \check\Lambda_1^N D
+ D_0^T ( \check\Lambda_1^N  +  \check\Lambda_2^N ) D_0 \\
\hspace{1.5cm}
 -  [ ( \check{S}^{N T}_1 + \check{S}_2^{NT} ) B + D^T \check\Lambda_1^N B_1 + D_0^T (\check\Lambda_1^N + \check\Lambda_2^N ) B_0 ] \Theta_2    \\
\hspace{1.5cm}
 - \Theta_2^T  [ B^T ( \check{S}^{N }_1 + \check{S}_2^{N} )
 +  B_1^T \check\Lambda_1^N D
 + B_0^T (\check\Lambda_1^N + \check\Lambda_2^N ) D_0 ]  \\
\hspace{1.5cm}
+ (D_0 - B_0 \Theta_2 )^T ( \check\Lambda_{22}^N
 + \check\Lambda_{12}^N  + \check\Lambda_{12}^{N T} ) (D_0 - B_0 \Theta_2
)   \\
\hspace{1.5cm} + \check{g}_{03}(N, \check\Lambda_2^N) , \\
\check{r}^N(T) = 0 ,
\end{cases}  \label{ODEcheckrN}
\end{align}
with
\begin{align}
& \check{g}_1(N, \check{\Lambda}_1^N, \check{\Lambda}_2^N )
=  ({1}/{N})\Big\{  \Theta^T B_0^T [ \check{\Lambda}_1^N + (1-1/N) \check{\Lambda}_2^N ] B_0 \Theta  \notag \\
& \qquad\qquad \qquad \quad
 +  [ \check{\Lambda}_1^N + (1-1/N) \check{\Lambda}_2^N ] G
+ G^T [ \check{\Lambda}_1^N + (1-1/N) \check{\Lambda}_2^N ]
 +Q^\Gamma  \Big\} ,\notag \\
 & \check{g}_2 (N,  \check{\Lambda}_2^N )
= - ( \Theta^T B_0^T \check{\Lambda}_2^N B_0 \Theta +  \check{\Lambda}_2^N G + G^T \check{\Lambda}_2^N  )/N ,   \notag \\
& \check{g}_{12}(N,  \check\Lambda_2^N) =  (- \Theta^T B_0^T \check\Lambda_2^N B_0 \Theta_1 +  \check\Lambda_2^N B \Theta_1 )/N ,   \notag  \\
 & \check{g}_{22}(N,  \check\Lambda_2^N ) =  -
  \Theta_1^T B_0^T \check{\Lambda}_2^N  B_0 \Theta_1/N   ,  \notag \\
 & \check{g}_{01} (N, \check\Lambda_2^N )
=  [ \Theta^T B_0^T \check\Lambda_2^N
 ( D_0 - B_0\Theta_2 )
+  \check\Lambda_2^N B\Theta_2]/N , \notag  \\
 & \check{g}_{02}(N,  \check\Lambda_2^N)
= [ - \Theta_1^T B_0^T \check\Lambda_2^N B_0 \Theta_2
+ \Theta_1^T B_0^T \check\Lambda_2^N D_0 ]/N ,  \notag  \\
 & \check{g}_{03}(N,  \check\Lambda_2^N )
= [ -  \Theta_2^T B_0^T \check\Lambda_2^N B_0 \Theta_2
 + {2} D_0^T \check\Lambda_2^N B_0 \Theta_2
 -  D_0^T \check\Lambda_2^N D_0]/N .
\notag
\end{align}

\begin{remark}
\label{rmk: sol check system}
Given $(\Lambda_1 , \Lambda_2 )$ on $[0, T]$, the system
\eqref{ODEcheckLam1N}--\eqref{ODEcheckrN} is a linear ODE system and  has a unique solution on $[0, T]$ for each $N \geq 1$.
\end{remark}

\begin{remark}\label{remark:checkLamC}
Let $\psi^N$ stand for any of the functions $ \check \Lambda^N_1$,
$ \check\Lambda^N_2$, $ \check\Lambda^N_{12}$, $ \check\Lambda^N_{22}$,
$\check S_1^N$, $\check S_2^N$ and $ \check r^N$. Due to the bounded coefficients in the ODE system, $\sup_{N\geq 1,\ 0\leq t \leq T} |\psi^N|   \leq C$ for some fixed constant $C$.
\end{remark}

\begin{remark} \label{remark:g1overn}
Let $h^N$ stand for any of the functions  $\check{g}_1$, $\check{g}_{2}$,
$\check{g}_{12}$, $\check{g}_{22}$,
$\check{g}_{01}$, $\check{g}_{02}$ and $\check{g}_{03}$. Then
  $\sup_{t\in [0,T]}|h^N(t)| =O(1/N)$.
 \end{remark}

\subsection{Upper bound of optimality gap}

Under Assumption~\ref{assm: sol Lambda Lambda_1} on
 $(\Lambda_1 , \Lambda_2)$,
 for all sufficiently large $N$  we can solve for $(\Lambda_1^N, \Lambda_2^N, S^N, r^N)$ according to Theorem~\ref{thm: NSAS} and Corollaries~
\ref{cor: sol Lambda_1^N Lambda_2^N} and \ref{cor: sol S^N r^N}.
For every such $N$, we  solve the system
\eqref{ODEcheckLam1N}--\eqref{ODEcheckrN}
for a unique solution $(\check\Lambda_1^N, \check\Lambda_2^N, \check\Lambda_{12}^N, \check\Lambda_{22}^N, \check{S}_1^N, \check{S}_2^N, \check{r}^N)$.

The following lemmas  estimate some difference terms relating the
low-dimensional functions
$(\Lambda_1^N, \Lambda_2^N, S^N, r^N)$ to $(\check\Lambda_1^N, \check\Lambda_2^N, \check\Lambda_{12}^N, \check\Lambda_{22}^N, \check{S}_1^N, \check{S}_2^N, \check{r}^N)$,
which will be used to estimate the optimality loss
$|J_{\rm soc}^{(N)}(U^o) - J_{\rm soc}^{(N)}({U}^d)|$ of the decentralized control $U^d$.

\begin{lemma}
\label{lm: diff Lambda_1}
$\sup_{t\in[0, T]}|\check\Lambda_1^N  - \Lambda_1^N  | = O(1/N)$.
\end{lemma}

\begin{proof}
By Corollary~\ref{cor: sol Lambda_1^N Lambda_2^N},
$\sup_{t\in[0, T]}|\Lambda_1^N(t) - \Lambda_1(t)| = O(1/N)$,
so it suffices to show that
$\sup_{t\in [0, T]}|\check\Lambda_1^N(t) - \Lambda_1(t)| = O(1/N)$.
Taking the difference of \eqref{ODELam1}
and \eqref{ODEcheckLam1N} gives
\begin{align}
\begin{cases}
 \tfrac{d}{dt} (\check{\Lambda}_1^N  - \Lambda_1 )
 =  - \Theta^T B_1^T (\check{\Lambda}_1^N - \Lambda_1 ) B_1 \Theta
 - (\check{\Lambda}_1^N - \Lambda_1 ) ( A - B \Theta )  \\
 \qquad\qquad\qquad \quad  - ( A -  B \Theta )^T (\check{\Lambda}_1^N - \Lambda_1 ) - \check{g}_1(N, \check{\Lambda}_1^N, \check{\Lambda}_2^N) , \\
 \check{\Lambda}_1^N(T) - \Lambda_1(T) =
 Q^\Gamma_f/N   .  \notag
\end{cases}
\end{align}
Then it follows that for all $t\in [0, T]$,
\begin{align}
 | \check\Lambda_1^N(t) - \Lambda_1(t) |
 \leq &   \int_t^T\{ ( | \Theta^T B_1^T|^2
 +  2 | A - B \Theta | )
 | \check{\Lambda}_1^N - \Lambda_1 |
 + | \check{g}_1(N, \check{\Lambda}_1^N, \check{\Lambda}_2^N) | \} ds   \notag \\
& + | Q^\Gamma_f  |/N  . \notag
\end{align}
By Remark \ref{remark:g1overn}, $\sup_t |\check{g}_1(N, \check\Lambda_1^N, \check\Lambda_2^N)| = O(1/N)$. The lemma follows from Gr\"onwall's lemma.
\end{proof}

\begin{lemma}
\label{lm: diff sum Lambda}
$\sup_{t\in[0, T]}
|\Lambda_1^N + \Lambda_2^N - ( \check\Lambda_1^N + \check\Lambda_2^N  + \check\Lambda_{12}^N + \check\Lambda_{12}^{N T}
 + \check\Lambda_{22}^N ) |  = O(1/N)$.
\end{lemma}
\begin{proof}
Define
$\Delta_{12}^N:=\Lambda_1 + \Lambda_2 - \check\Lambda_1^N - \check\Lambda_2^N - \check\Lambda_{12}^N - \check\Lambda_{12}^{N T}
 -  \check\Lambda_{22}^N$.
Since $\sup_{t\in [0, T]} (|\Lambda_1^N - \Lambda_1 | + |\Lambda_2^N - \Lambda_2 | ) = O(1/N)$ by Corollary~\ref{cor: sol Lambda_1^N Lambda_2^N},
it suffices to show that
$\sup_{t\in [0, T]}|  \Delta_{12}^N |  = O(1/N)$.
We combine \eqref{ODEcheckLam1N}--\eqref{ODEcheckLam22N} and
\eqref{ODELam2} to get the ODE
\begin{align}
 & \frac{d}{dt}  \Delta_{12}^N    =  ( \Theta + \Theta_1 )^T
  [ B_1^T ( \check{\Lambda}_1^N - \Lambda_1 ) B_1 -  B_0^T \Delta_{12}^N B_0 ]    (\Theta + \Theta_1 ) \notag  \\
& \qquad\qquad - \Delta_{12}^N [ A + G - B ( \Theta + \Theta_1 ) ]
   - [ A + G - B( \Theta + \Theta_1 ) ]^T  \Delta_{12}^N \notag \\
 & \qquad\qquad + \check{g}_1 + \check{g}_2 + \check{g}_{12} + \check{g}_{12}^T +  \check{g}_{22} ,  \notag \\
 & \Delta_{12}^N(T)
=  - Q^\Gamma_f /N .       \notag
\end{align}
Since $\sup_{t\in[0, T]} | \check\Lambda_1^N - \Lambda_1 | = O(1/N)$ by the proof of
Lemma~\ref{lm: diff Lambda_1}
and $\sup_t|\check{g}_1 + \check{g}_2 + \check{g}_{12} + \check{g}_{12}^T
+  \check{g}_{22}| = O(1/N)$ by Remark \ref{remark:g1overn},  the desired result follows from Gr\"onwall's lemma, in the same manner as in the proof of Lemma~\ref{lm: diff Lambda_1}.
\end{proof}

\begin{lemma}
\label{lm: diff S}
$\sup_{t\in[0, T]} |S^N - \check{S}_1^N - \check{S}_2^N| = O(1/N)$.
\end{lemma}
\begin{proof}
By Corollary~\ref{cor: sol S^N r^N},
it suffices to show that
$\sup_{t\in[0, T]} |S - \check{S}_1^N - \check{S}_2^N| = O(1/N)$.
Combining \eqref{ODE check S_1}, \eqref{ODE check S_2}
and \eqref{ODESlim} gives
\begin{align}
  \frac{d}{dt} ( S - \check{S}_1^N - \check{S}_2^N)
 = & - [A+G - B(\Theta + \Theta_1) ]^T (S - \check{S}^N_1 - \check{S}^N_2)
\notag \\
& + (\Theta + \Theta_1)^T B_1^T (\Lambda_1 - \check\Lambda_1^N) (D - B_1 \Theta_2) \notag \\
& + (\Theta + \Theta_1)^T B_0^T \Delta^N_{12} (D_0 - B_0 \Theta_2)
+ \Delta^N_{12} B \Theta_2 \notag \\
 &  + \check{g}_{01}(N, \check\Lambda_2^N)
+ \check{g}_{02}(N, \check\Lambda_2^N) ,
\notag
\end{align}
where
$ S(T) - \check{S}_1^N(T) - \check{S}_2^N(T) = 0$.
With the estimates of $\Lambda_1 - \check\Lambda_1^N$ and
$\Delta_{12}^N$ obtained in the proofs of Lemmas \ref{lm: diff Lambda_1} and \ref{lm: diff sum Lambda}, respectively, and $\sup_t|\check{g}_{0k}(N, \check\Lambda_2^N)| = O(1/N)$, $k=1, 2$ in Remark \ref{remark:g1overn},
the desired result follows from Gr\"{o}nwall's lemma.
\end{proof}

\begin{lemma}
\label{lm: diff r}
$\sup_{t\in [0, T]} | r^N - \check{r}^N | = O(1/N)$.
\end{lemma}
\begin{proof}
By Corollary~\ref{cor: sol S^N r^N}, it suffices to show that
$\sup_{t\in [0, T]} | r - \check{r}^N| = O(1/N)$, where $r$ is the unique solution of \eqref{ODErlim}.
Combining \eqref{ODErlim} and \eqref{ODEcheckrN} gives
\begin{align}
& \frac{d}{dt}  ( r - \check{r}^N )
=  \Theta_2^T [B_1^T (\check\Lambda_1^N - \Lambda_1 ) B_1 - B_0^T   \Delta_{12}^N  B_0] \Theta_2  \notag \\
&\qquad\qquad + [ ( S^T - \check{S}_1^{N T} - \check{S}_2^{N T}  ) B
 + D^T ( \Lambda_1 - \check\Lambda_1^N  ) B_1
  + D_0^T \Delta_{12}^N  B_0 ] \Theta_2   \notag \\
&\qquad\qquad  + \Theta_2^T [ B^T(S - \check{S}^N_1 - \check{S}^N_2) + B_1^T (\Lambda_1 - \check\Lambda_1^N) D  + B_0^T \Delta_{12}^N D_0 ]
\notag \\
&\qquad\qquad + D^T (\check\Lambda_1^N - \Lambda_1 ) D - D_0^T \Delta_{12}^N  D_0 + \check{g}_{03}(N, \check\Lambda_2^N)  ,  \notag \\
& r(T) - \check{r}^N(T) =  0 . \notag
\end{align}
With the estimates of $\Lambda_1 - \check\Lambda_1^N$, $\Delta_{12}^N$ and $S - \check{S}_1^N - \check{S}_2^N$ obtained in the proofs of Lemmas \ref{lm: diff Lambda_1}, \ref{lm: diff sum Lambda} and \ref{lm: diff S},
and $\sup_t|\check{g}_{03}(N, \check\Lambda_2^N)| = O(1/N)$,
we obtain the desired result.
\end{proof}

\begin{proof}[Proof of Theorem~\ref{thm: performance analysis}]
We have
\begin{align}
 & J_{\rm soc}^{(N)}(U^d) - J_{\rm soc}^{(N)} (U^o) \nonumber \\
 =&{\mathbb E} [ \check{V}(0, X(0), \overline{X}(0)) -  V(0, X(0))  ]
\notag \\
 = &\mathbb{E} [ X^T(0) (  \check{\mathbf{P}}_1(0) -  \mathbf{P}(0)  ) X(0)
+ 2 X^T(0) \check{\mathbf{P}}_{12}(0) \overline{X}(0)
  + \overline{X}^T(0) \check{\mathbf{P}}_2 (0) \overline{X}(0) ]
\label{JNU-JNcheckU1} \\
& + 2 \mathbb{E} [  X^T(0) \check{\mathbf{S}}_1(0) + \overline{X}^T(0)
\check{\mathbf{S}}_2(0) - X^T(0) \mathbf{S}(0) ]
+ \check{\mathbf{r}}(0) -  \mathbf{r}(0)  . \notag
\end{align}
The linear term on the right hand side of \eqref{JNU-JNcheckU1} may be written as
\begin{align}
 & 2 \mathbb{E} [ X^T(0) \check{\mathbf{S}}_1(0) + \overline{X}^T(0) \check{\mathbf{S}}_2(0) - X^T(0) \mathbf{S}(0)  ]  \notag \\
 & =  2 \sum_{i=1}^N \mathbb{E} [   X_i^T(0) \check{S}_1^N(0)
 + \overline{X}^T(0) \check{S}_2^N(0) - X_i^T(0) S^N(0) ] \notag \\
 & =  2 N  \mu_0^T ( \check{S}_1^N(0) +  \check{S}_2^N(0) - S^N(0)  )  .
\label{S-checkS}
\end{align}
The quadratic term on the right hand side of \eqref{JNU-JNcheckU1} may be
written as
\begin{align}
&  \mathbb{E} [ X^T(0) (  \check{\mathbf{P}}_1(0) - \mathbf{P}(0)  ) X(0)
 + 2 X^T(0) \check{\mathbf{P}}_{12}(0) \overline{X}(0)
  + \overline{X}^T(0) \check{\mathbf{P}}_2(0) \overline{X}(0) ]
 \notag \\
& = \sum_{i=1}^N \mathbb{E} [ X_i^T(0)  ( \check\Lambda_1^N(0)
 - \Lambda_1^N(0)) X_i(0) ]
+ \sum_{i \neq j = 1}^N \frac{1}{N} \mathbb{E} [ X_i^T(0) ( \check\Lambda_2^N(0) - \Lambda_2^N(0)  ) X_j(0) ]  \notag \\
& \quad  +  \sum_{i=1}^N \mathbb{E} [ X_i^T(0) \check\Lambda_{12}^N (0)
\overline{X}(0) ]
 +  \sum_{i=1}^N \mathbb{E} [ \overline{X}^T(0) \check\Lambda_{12}^{N T}(0) X_i(0) ]
 + N \mathbb{E} [ \overline{X}^T(0)  \check\Lambda_{22}^N (0) \overline{X}(0) ] \notag \\
 & =\sum_{i=1}^N \Tr [   ( \check\Lambda_1^N(0)
 - \Lambda_1^N(0))  \Sigma_0^{i}   ]+ \nonumber\\
 &\quad + N  \mu_0^T [  \check\Lambda_1^N(0) + \check\Lambda_2^N(0)
  + \check\Lambda_{12}^N (0) + \check\Lambda_{12}^{N T}(0)  +  \check\Lambda_{22}^N (0)
 - \Lambda_1^N(0)  - \Lambda_2^N(0) ] \mu_0    \notag \\
& \quad
 -   \mu_0^T[\check\Lambda_2^N(0) - \Lambda_2^N(0)]\mu_0  \nonumber
  \\
  & \eqqcolon \zeta_0^N .   \nonumber
\end{align}
Substituting \eqref{S-checkS} and $\zeta_0^N$ into \eqref{JNU-JNcheckU1} gives
\begin{align}
 J_{\rm soc}^{(N)}(U^d) - J_{\rm soc}^{(N)}(U^o)
 & =\zeta_0^N
   +    2 N  \mu_0^T (  \check{S}_1^N(0)
 + \check{S}_2^N(0) - S^N(0) ) + \check{\mathbf{r}}(0) - \mathbf{r}(0) . \notag
\end{align}
By Lemma \ref{lm: diff Lambda_1} and Assumption \ref{assumption:mean}, $|\sum_{i=1}^N \Tr [   ( \check\Lambda_1^N(0)
 - \Lambda_1^N(0))  \Sigma_0^{i}   ]|= O(1)$.
By Corollary \ref{cor: sol Lambda_1^N Lambda_2^N} and Remark \ref{remark:checkLamC},  $|\check\Lambda_2^N(0) - \Lambda_2^N(0)| = O(1)$.  The two
upper bounds combined with Lemma \ref{lm: diff sum Lambda} imply that  $|\zeta_0^N| =O(1)$.  Recalling  Lemmas
\ref{lm: diff S} and \ref{lm: diff r}, it follows that
 $|J^{(N)}_{\rm soc}(U^d) - J^{(N)}_{\rm soc}(U^o)| = O(1)$.
Furthermore, the optimality of $U^o$ implies that $ J^{(N)}_{\rm soc}(U^d) - J^{(N)}_{\rm soc}(U^o)\ge 0$, and thus the desired result follows.
\end{proof}

\subsection{Performance comparison with the mean field game}
\label{sec:sub:compSocMFG}
The agents in the social optimization problem are cooperative with a common objective. A different solution notion is to solve a mean field game where   each agent optimizers for its own interest; this has been developed in a companion paper \cite{HY2020b}. This subsection compares the two solutions by demonstrating the efficiency gain of  social optimization with respect  to the mean field game.

Let $U^o=(u_1^T, \cdots, u_N^T)^T$ denote the social optimal control and  $U^g = (u_1^{gT}, \cdots, u_N^{gT})^T$ the set of Nash equilibrium strategies. For simplicity, we consider the  model with $D= D_0 = 0$ in \eqref{X_i}.
For the comparison, we further assume the mean and covariance matrix  of the initial states
  $\{X_i(0): 1\leq i \leq N\}$  satisfy \eqref{meancovX0}.

When all agents take the social optimal control $U^o$,   the asymptotic per agent cost based on \eqref{jiuoN} is defined as \begin{align}
  \bar{J}_{i, \rm{soc}}
& \coloneqq \lim_{N\to\infty} (1/N) \mathbb{E} V(0, X(0))  \notag \\
& = \mathbb{E}[X_1^T(0) \Lambda_1(0) X_1(0) + X_1^T(0)\Lambda_2(0) X_2(0)  ]  \notag \\
& = \Tr[\Lambda_1(0) \Sigma_0] +\mu_0^T [\Lambda_1(0)+\Lambda_2(0)]\mu_0 .\label{JiUinfN}
\end{align}
When $U^d$ instead of $U^o$ is applied, by Theorem \ref{thm: performance analysis},  the per agent cost also tends to   $\bar{J}_{i, \rm{soc}}$ as
$N\to \infty$.

When all agents take the set of Nash equilibrium strategies $U^g$, let $V_i^g$ denote the value function of agent $\mathcal{A}_i$. The asymptotic per agent cost is  defined as
\begin{align}
 \bar{J}_{i, \rm{mfg}}
\coloneqq & \lim_{N\to\infty} \mathbb{E}V_i^g(0, X(0))    \notag \\
= & \mathbb{E} \left[ X_1^T(0)   \Lambda_1^{g}(0)    X_1(0)
 +  X_1^T(0) ( \Lambda_2^{g}(0) + \Lambda_2^{gT}(0) + \Lambda_4^{g}(0)  )
X_2(0)  \right]  \notag \\
=& \Tr[  \Lambda_1^g(0) \Sigma_0 ]+ \mu_0^T[\Lambda_1^g(0) +\Lambda_2^{g}(0) + \Lambda_2^{gT}(0) + \Lambda_4^{g}(0)  ] \mu_0  .
\label{JiUginfN}
\end{align}
The ODEs of $(\Lambda_1^g, \Lambda_2^g, \Lambda_3^g, \Lambda_4^g)$ are given in Appendix~\ref{appendix: ODEs Lamg}.  We have $\Lambda_1 = \Lambda_1^g$ on $[0, T]$, since $\Lambda_1$ and $\Lambda_1^g$ satisfy the same Riccati ODE with the same terminal condition. The per agent cost for the
decentralized $\epsilon$-Nash equilibrium strategies (see \cite{HY2020b}) has the same limit  $\bar{J}_{i, \rm{mfg}}$ as $N\to \infty$.
By  \eqref{JiUinfN} and \eqref{JiUginfN},
we calculate
\begin{align}
 & \bar{J}_{i,{\rm mfg}} - \bar{J}_{i,{\rm soc}}    = \mu_0^T[ \Lambda_1^g(0) + \Lambda_2^{g}(0) + \Lambda_2^{gT}(0) + \Lambda_4^{g}(0)
 - \Lambda_3 (0)  ] \mu_0 . \notag
\end{align}
Since $(1/N) \mathbb{E} V(0, X(0)) \le \mathbb{E}V_i^g(0, X(0)) $ for each $N$, we have $\bar{J}_{i,{\rm mfg}} - \bar{J}_{i,{\rm soc}}\ge 0$.

\subsection{Comparison with mean field type control}
\label{sec:sub:compMFTC}

We describe  an  application of  mean field type optimal control to mean-variance portfolio selection.  The state process is given by
\begin{align}
dX(t) = [\rho X(t) +(\alpha- \rho) u(t)] dt+\sigma u(t) dW(t), \nonumber
\end{align}
where $X(0)=x_0>0$. For simplicity, we consider one bond and one stock with constant parameters. In the above, $X(t)$ is the wealth; $u(t)$ is the amount allocated to the stock;  $\rho$ is the interest rate of the bond; $\alpha>\rho $ is the appreciation rate of the stock; and $\sigma>0$ is the volatility of  the stock. The above state equation has an obvious generalization by considering time-dependent parameters and more than one
stock.  The cost is
\begin{align}
J(u) &= \frac{\gamma}{2} Var(X(T))-\mathbb{E} X(T) \nonumber\\
 & =\mathbb{E} \Big(\frac{\gamma}{2} X^2(T) -X(T)\Big) -\frac{\gamma}{2} ( \mathbb{E} X(T))^2,\quad \gamma>0.
\end{align}
The mean-variance portfolio selection problem has been solved for a more general case of multiple stocks \cite{ZL2000}, \cite[Chap. 6]{YZ1999}. The method there is to solve a family of problems by dynamic programming.
Alternatively, \cite{AD2011} uses the stochastic  maximum principle to derive the solution as follows. Denote
\begin{align*}
&(\rho -\alpha)^2 A_t-(2\rho A_t+\dot{A}_t)\sigma^2 =0,\\
&\rho C_t +\dot{C}_t=0,
\end{align*}
where $A_T=\gamma$ and $C_T=1$.
Then
$$
A_t= \gamma e^{(2\rho -\lambda )(T-t) },  \qquad
C_t= e^{\rho (T-t)} ,
$$
where
$\lambda = (\rho -\alpha)^2/\sigma^2$.
The optimal control law is
\begin{align}
\hat u(t) &= \frac{\alpha-\rho }{\sigma^2}[ C_tA_t^{-1}  - (X(t) -\mathbb{E}  X(t))] . \label{uhatmv1}
 \end{align}
After solving $\mathbb{E}X(t)$ from a linear ODE, one obtains
\begin{align}
\hat u(t) =\frac{\alpha-\rho }{\sigma^2}\Big[ x_0 e^{\rho t} + \frac{1}{\gamma }  e^{\lambda T -\rho(T-t)}-X(t) \Big]. \label{uhatmv2}
\end{align}

For the social optimization problem we consider the scalar model which has the state equations
\begin{align*}
dX_i(t) &= [A X_i(t)+Bu_i(t)] dt +B_1 u_i(t) dW_i(t), \quad 1\le i\le N,\\
 & :=[\rho X_i(t) + (\alpha-\rho) u_i(t)] dt + \sigma u_i(t) dW_i(t)
\end{align*}
and individual costs
\begin{align}
J_i= \frac{\gamma }{2}\mathbb{E} |X_i(T)- X^{(N)}(T)|^2 -\mathbb{E} X_i(T), \quad 1\le i\le N.\label{Jimv}
\end{align}
The social cost is $J_{\rm soc}^{(N)}=\sum_{i=1}^NJ_i$.
Suppose all agents have identical initial state $x_0$.
The mean-variance portfolio selection problem in \cite{FL16} is solved by means of solving the LQ social optimization problem and passing to an infinite population.
Below we will tailor the results in previous sections to this particular model.

To adapt to the costs \eqref{Jimv}, we slightly modify the individual costs in \eqref{J soc} by replacing the terminal cost with the new one
$$\mathbb{E}\Big\{[X_i(T) - \Gamma_f X^{(N)}(T)]_{Q_f}^2 +2K^T X_i(T)\Big\},$$
where $K\in \mathbb{R}^n$. Accordingly, we take ${\mathbf S}(T)= [K^T, \cdots, K^T]^T$ in \eqref{ODE S},  $S^N(T)= K $ in \eqref{ODESN}, and $S(T)=
K$ in \eqref{ODESlim}.

Matching the notation in \eqref{X_i}--\eqref{J soc}, we have
$Q=0$, $\Gamma=0,$ $R=0$, $Q_f = \gamma/2$,
and $\Gamma_f=1$.
We further set $K=-1/2$.
By \eqref{ODELam1},  we have
\begin{align}
{\dot \Lambda}_1 &= \Lambda_1 B(B_1\Lambda_1 B_1)^{-1} B \Lambda_1 - 2 A \Lambda_1 \nonumber \\
&=  \frac{(\alpha-\rho)^2}{\sigma^2} \Lambda_1 -2\rho \Lambda_1, \nonumber
\end{align}
where  $ \Lambda_1(T)=\gamma/2$. We obtain
$ \Lambda_1 (t)=(\gamma/2) e^{(2\rho -\lambda)(T-t) }$.
Next by \eqref{ODELam3}, $\Lambda_3=\Lambda_1+\Lambda_2$ satisfies
\begin{align}
\dot{\Lambda}_3=\frac{(\alpha-\rho)^2}{\sigma^2}
\Lambda_1^{-1} \Lambda_3^2  -2 \rho \Lambda_3, \nonumber
\end{align}
where $\Lambda_3(T)=0$. Hence $\Lambda_3=0$ and $\Lambda_2=-\Lambda_1$.
Now \eqref{ODESlim} reduces to
\begin{align}
\dot{S}=-\rho S , \notag
\end{align}
where $S(T)=K= -1/2$.

Recalling $\Lambda_2=-\Lambda_1$, we check \eqref{check u_i} and determine
\begin{align}
\Theta= \frac{\alpha -\rho}{\sigma^2},\quad \Theta_1=-\Theta, \quad \Theta_2=\frac{ \alpha-\rho}{\sigma^2} S \Lambda_1^{-1}.  \nonumber
\end{align}
The decentralized individual control is
\begin{align}
\check u_i=\frac{\alpha -\rho}{\sigma^2} [-  S(t)(\Lambda_1(t))^{-1} - (X_i(t)-\overline  X(t))  ], \quad 1\le i\le N, \label{uchecksoc1}
\end{align}
where $\overline X(t) =\mathbb{E} X_i(t) $. Clearly $-S(t) (\Lambda_1(t))^{-1}= C_t A_t^{-1} $ for all $t\in [0,T]$.

The two control laws \eqref{uhatmv1} and \eqref{uchecksoc1}   have the same form except for different interpretations of the mean term.
It is known that $\hat u$  in \eqref{uhatmv1} is not time consistent \cite{AD2011}. Given the value of $X(t_0)$ at $t_0>0$, the portfolio optimization problem re-solved on $[t_0, T]$ will generate  a  different strategy. However, for the mean field social optimization problem, the set of controls is time consistent.
We may think of an infinite population. Then given the available
states $X_i(t_0)$, $i\ge 1$, the optimization problem on $[t_0, T]$ will use  $\overline X$ as the restriction of
$\{\overline X(t), 0\le t\le T\}$ on $[t_0, T]$. The re-solved control is still given by  \eqref{uchecksoc1}.

\section{Numerical examples}
\label{sec: num}

This section presents numerical examples to illustrate  asymptotic solvability of  social optimization problems and examine performance of the associated control laws. The ODE systems are solved  using
MATLAB ODE solver \texttt{ode45}.

\subsection{Asymptotic solvability}

We consider three examples for \eqref{X_i}--\eqref{J soc}.

\begin{example}
\label{ex: Lambda solvable}
The parameter values are
$A=1$, $B=1$, $B_0=B_1=0.2$, $D_0=D=0$, $G=2$, $Q=4$, $Q_f = 2$, $R=1$,
$\Gamma=0.1$, $\Gamma_f = 0.1$, and $T=2$.
 Since  $R>0$,
$(\Lambda_1, \Lambda_2)$ has a  global solution on $[0, T]$,
implying that the social optimization problem has asymptotic solvability on $[0, T]$.  The solution of $(\Lambda_1, \Lambda_2)$ is  shown  in Fig.~\ref{fig: Lam12sol} (left panel).
\end{example}
\begin{example}
\label{ex: Lam12Rneg}
The parameter values are $A=-4$, $B=1$, $B_0=-2$, $B_1 =4$, $D_0=D=0$, $G=1$, $Q=Q_f=1$, $R=-1$, $\Gamma=4$, $\Gamma_f=2$, and $T=2$.
  Fig.~\ref{fig: Lam12sol} (middle panel)  shows that $(\Lambda_1, \Lambda_2)$ has a  global solution on $[0, T]$, suggesting that the social optimization problem has asymptotic solvability on $[0, T]$.
\end{example}

\begin{example}
\label{ex: Lam12explosion}
The parameter values are
$A=30$, $B=1$, $B_0=B_1=0.2$, $D_0=D=0$, $G=2$, $Q=-30$, $Q_f = 3$, $R=1.5$,
$\Gamma=0.1$, $\Gamma_f = 0.1$, and $T=2$.
Fig.~\ref{fig: Lam12sol} (right panel) shows that
 $(\Lambda_1, \Lambda_2)$ does not have a global solution on $[0, T]$.
Thus the social optimization problem does not have asymptotic solvability.
\end{example}

\begin{figure}[t]
\begin{center}
\begin{tabular}{ccc}
\hspace{-0.6cm}
\psfig{file=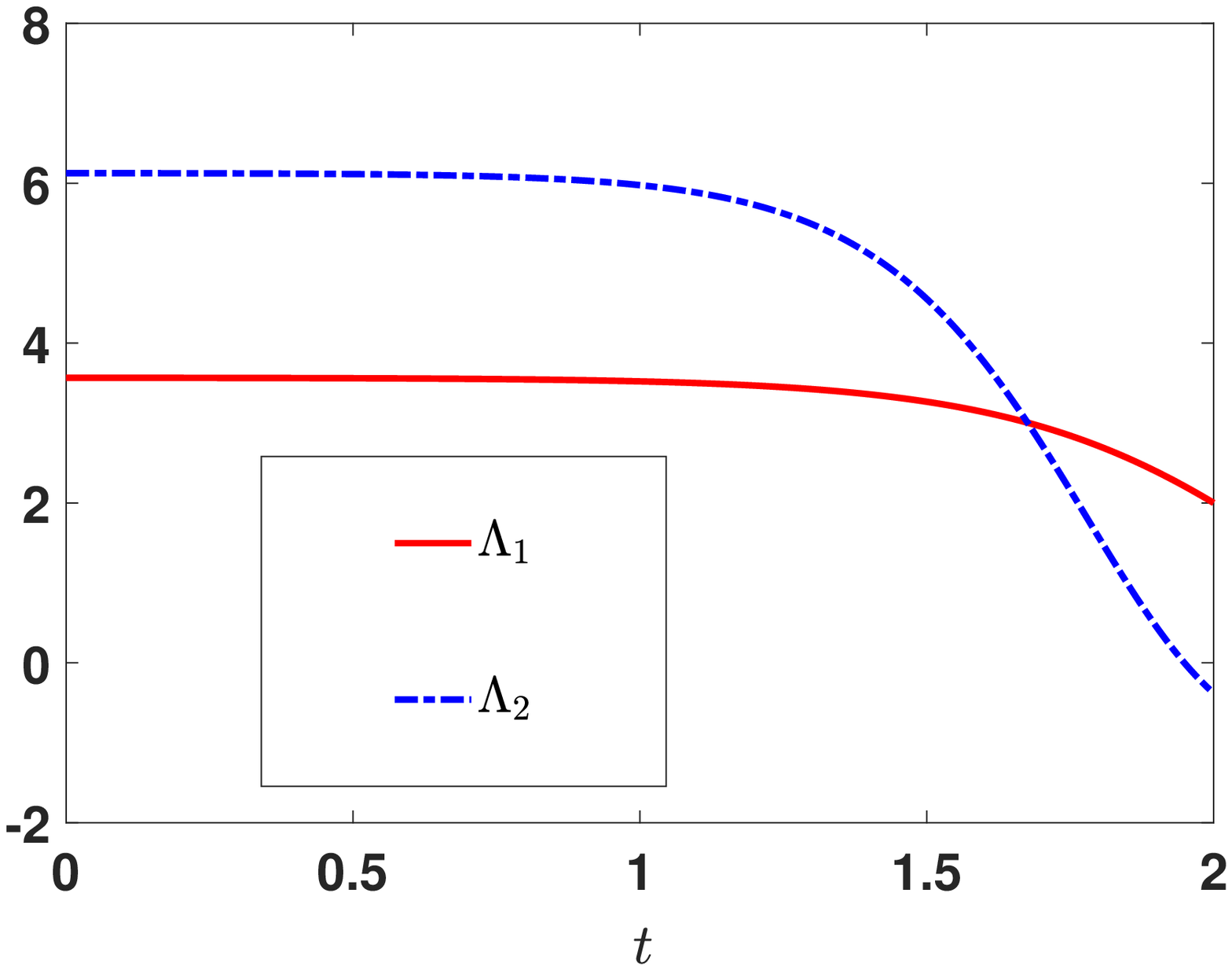, width=6.6cm, height=4cm}&
\hspace{-3cm}
\psfig{file=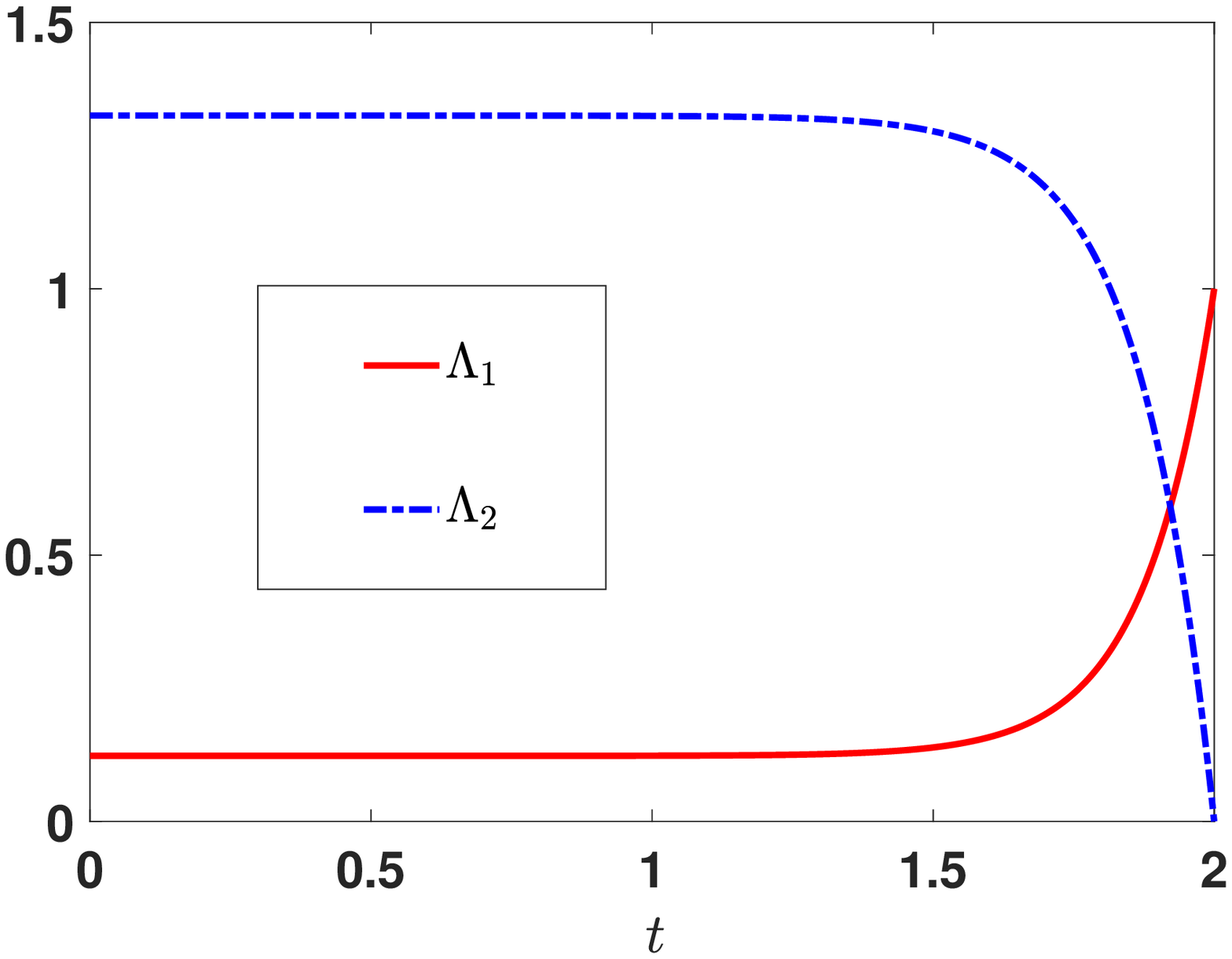, width=6.6cm, height=4cm}&
\hspace{-3cm}
\psfig{file=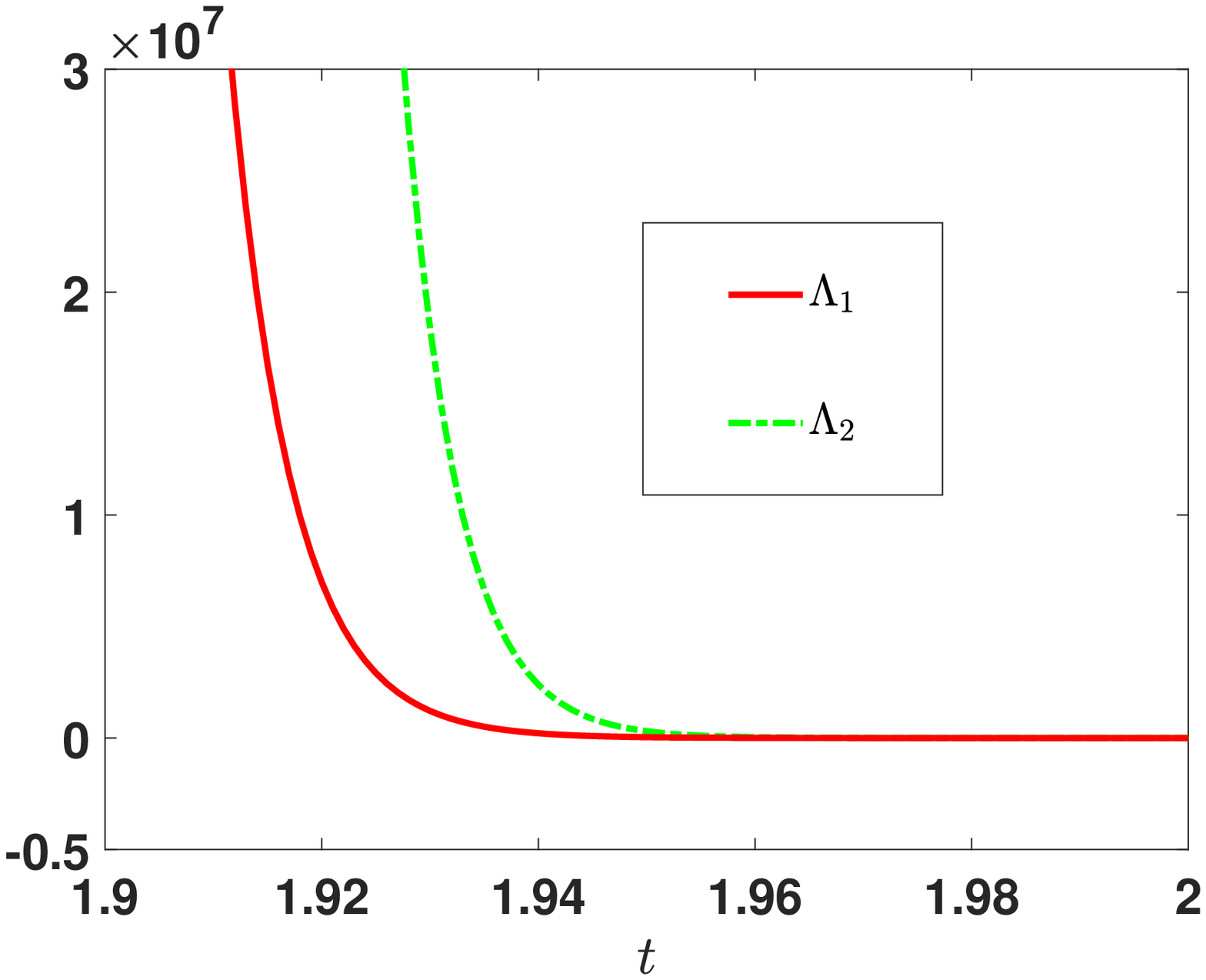, width=6.6cm, height=4cm}
\end{tabular}
\end{center}
\caption{Solvability of $(\Lambda_1, \Lambda_2)$ on $[0, T]$ with $T=2$.
Left panel: Example 1 with $R>0$.
Middle panel:  Example 2 with $R<0$.
Right panel: Example 3. }
\label{fig: Lam12sol}
\end{figure}


\subsection{Performance}

To  numerically compare $J_{\rm soc}^{(N)}(U^o)$ and  $J_{\rm soc}^{(N)}({U}^d)$, we need to solve the system~\eqref{ODELam12N} for $(\Lambda_1^N,
\Lambda_2^N)$ associated with the centralized control $U^o$ and next
the system \eqref{ODEcheckLam1N}--\eqref{ODEcheckLam22N}
for $(\check\Lambda_1^N, \check\Lambda_2^N, \check\Lambda_{12}^N, \check\Lambda_{22}^N)$ associated with the decentralized control $U^d$.
By  Theorem~\ref{thm: NSAS}, if the Riccati ODE system consisting of
\eqref{ODELam1} and \eqref{ODELam3} has a solution on $[0, T]$, then \eqref{ODELam12N} has a solution on $[0, T]$ for all  sufficiently large $N$,
and so does the system \eqref{ODEcheckLam1N}--\eqref{ODEcheckLam22N}.

Recall the necessary and sufficient condition in Section~\ref{sec:sub:solLam} for the solvability of \eqref{ODELam1} and \eqref{ODELam3}.
We will take  $Q$, $Q_f \geq 0$, and $R> 0$, under which \eqref{ODELam1} and \eqref{ODELam3} have a unique solution on $[0, T]$ \cite[Theorem 6.7.2]{YZ1999}.
With the same parameter values as in Example~\ref{ex: Lambda solvable}, we numerically solve the systems
\eqref{ODELam12N},
\eqref{ODELam1}--\eqref{ODELam2},
and \eqref{ODEcheckLam1N}--\eqref{ODEcheckLam22N}.
The initial conditions are given as $X_i(0)=1$ for all $i\geq 1$, and $\overline{X}(0)=1$.

 Fig.~\ref{fig: performance analysis} (left panel) shows that the
  difference ($\check\Lambda_{12}^N= \check\Lambda_{12}^{N T}$ for the scalar case)
\begin{align}
[ \check\Lambda_1^N(0) + \check\Lambda_2^N(0) + \check\Lambda_{12}^N(0) +
\check\Lambda_{12}^{N T}(0) + \check\Lambda_{22}^N(0)] - [\Lambda_1^N(0) + \Lambda_2^N(0)]  \notag
\end{align}
approaches $0$ as $N\to\infty$, as asserted by Lemma~\ref{lm: diff sum Lambda}.
Fig.~\ref{fig: performance analysis} (right panel) shows that the difference  $J_{\rm soc}^{(N)}({U}^d)-J_{\rm soc}^{(N)}(U^o)$ remains bounded as
$N$ increases.

\begin{figure}[t]
\begin{center}
\begin{tabular}{cc}
\psfig{file=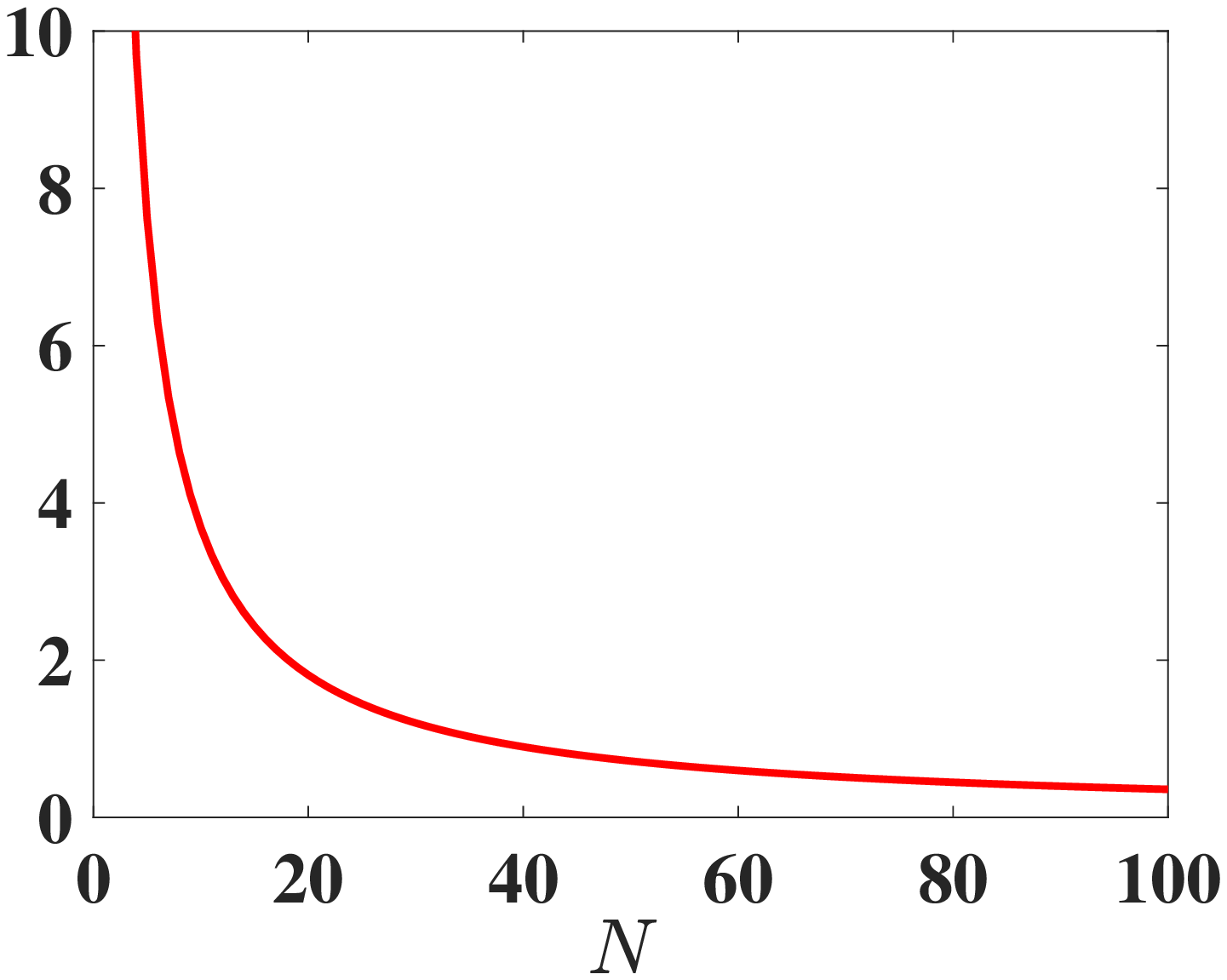, width=6.6cm, height=4cm}&
\hspace{-1cm}
\psfig{file=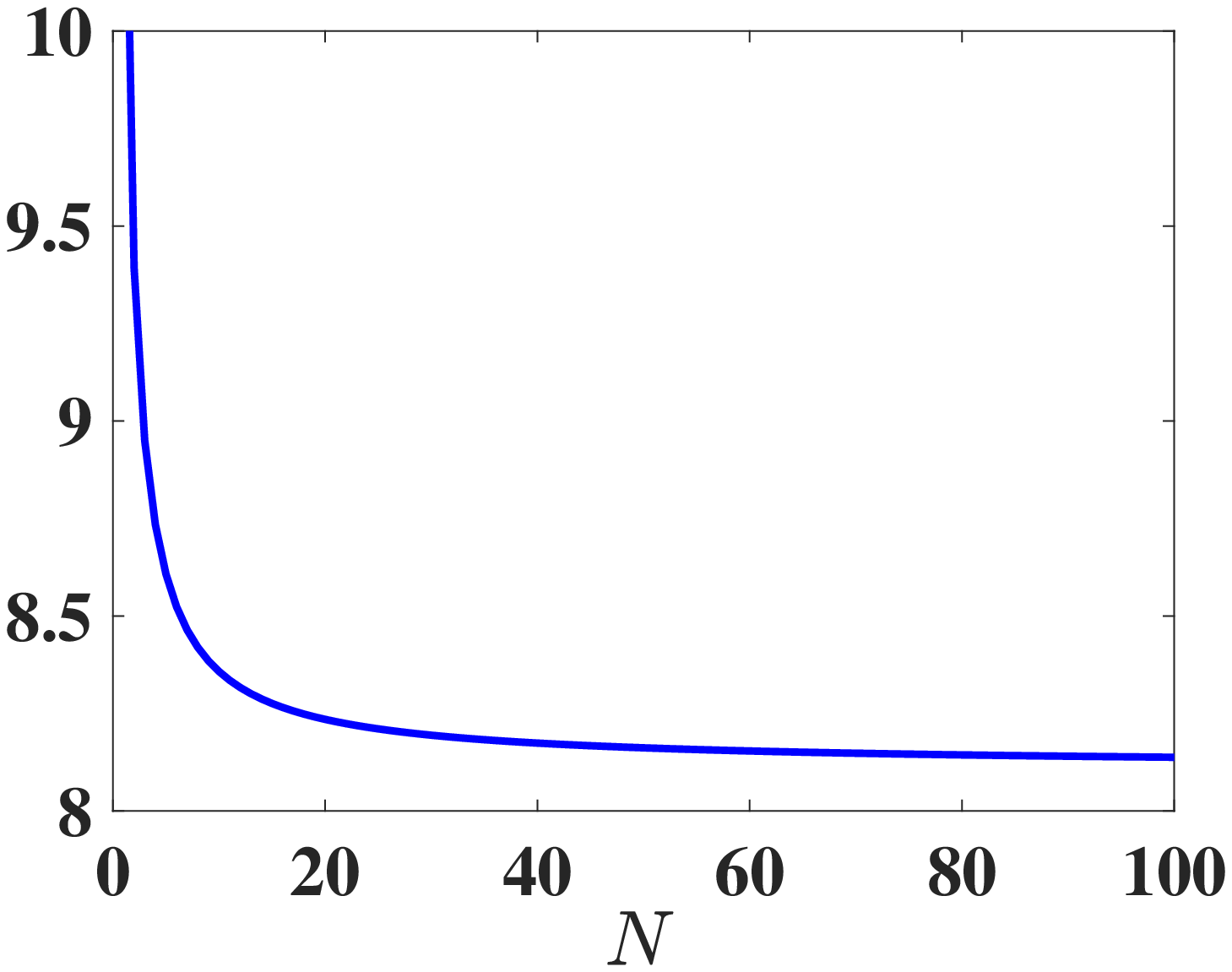, width=6.6cm, height=4cm}
\end{tabular}
\end{center}
\caption{Left panel: The difference $\check\Lambda_1^N + \check\Lambda_2^N + \check\Lambda_{12}^N + \check\Lambda_{12}^{N T} + \check\Lambda_{22}^N - (\Lambda_1^N + \Lambda_2^N)$ evaluated at $t=0$ for  $N\geq 1$. Right panel: The difference $J_{\rm soc}^{(N)}({U}^d) - J_{\rm soc}^{(N)}(U^o)$ for  $N\geq 1$. }
\label{fig: performance analysis}
\end{figure}

\subsection{Comparison between the social optimum and the mean field
equilibrium}
\label{section: compare U and U^g}

We use the model in Example ~\ref{ex: Lambda solvable} to  compare the per agent  costs $\bar{J}_{i,{\rm soc}}$ and $\bar{J}_{i,{\rm mfg}}$ for social optimization and   the mean field game, respectively. The initial states $X_i(0)$ have mean $\mu_0$ and variance $\Sigma_0$.

Fig.~\ref{fig: Lamg3} compares
$\Lambda_1^g + \Lambda_2^g + \Lambda_2^{gT} + \Lambda_4^g$ and $\Lambda_3$ on $[0, T]$ ($\Lambda_2^g = \Lambda_2^{gT}$ for the scalar case).
Note that $\Lambda_3= \Lambda_1+\Lambda_2$ on $[0,T]$.   Since
\begin{align}
\bar{J}_{i,{\rm mfg}}-\bar{J}_{i,{\rm soc}} = (\Lambda_1^g(0) +\Lambda_2^{g}(0) + \Lambda_2^{gT}(0) + \Lambda_4^{g}(0)-\Lambda_3(0)  ) \mu_0^2, \nonumber
\end{align}
 Fig.~\ref{fig: Lamg3} confirms that
the per agent cost of the  social optimal control is lower than that of the mean field  equilibrium strategy.

\begin{figure}[t]
\begin{center}
\begin{tabular}{cc}
\psfig{file=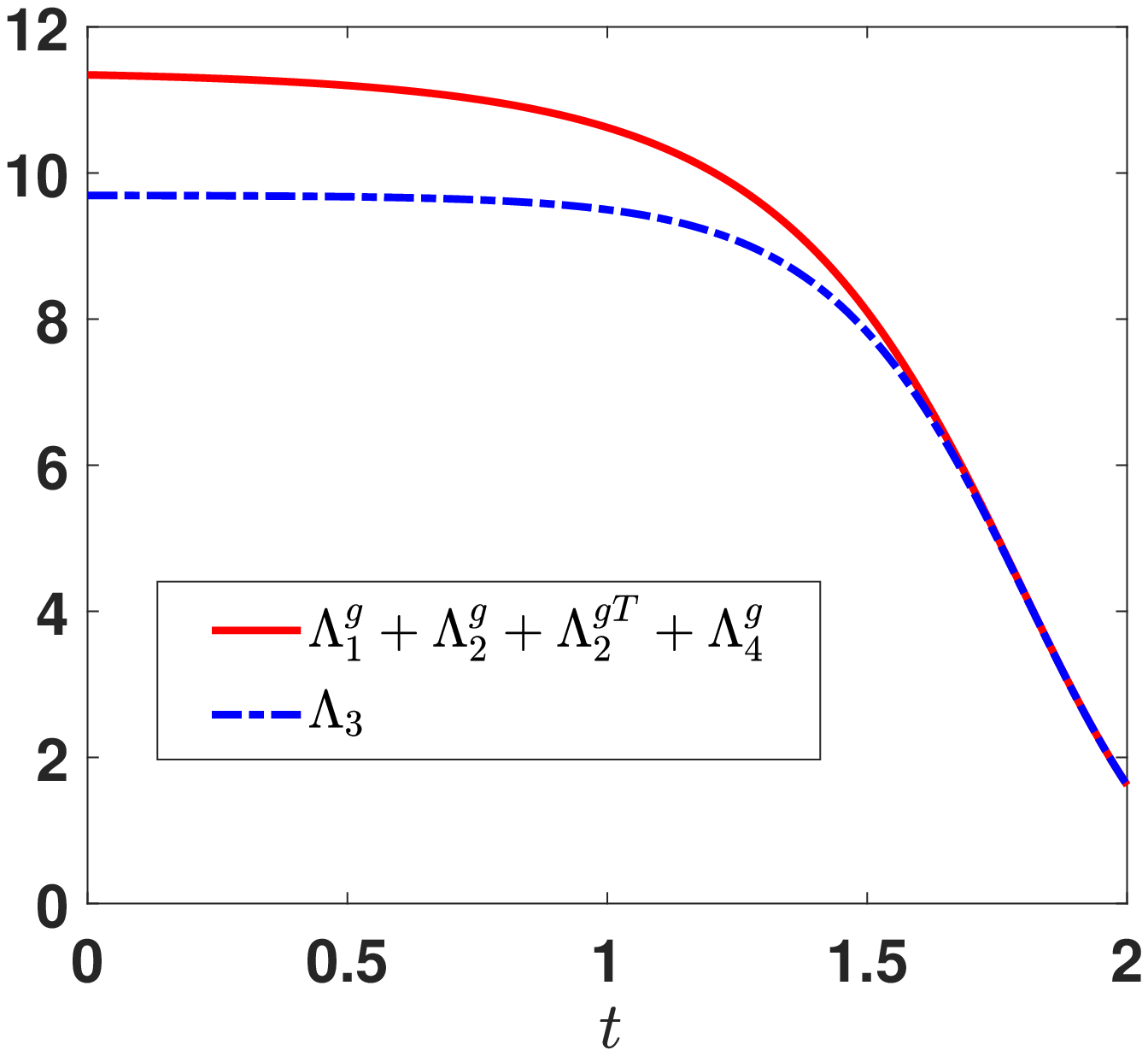, width=6.6cm, height=4cm}&
\hspace{-1cm}
\psfig{file=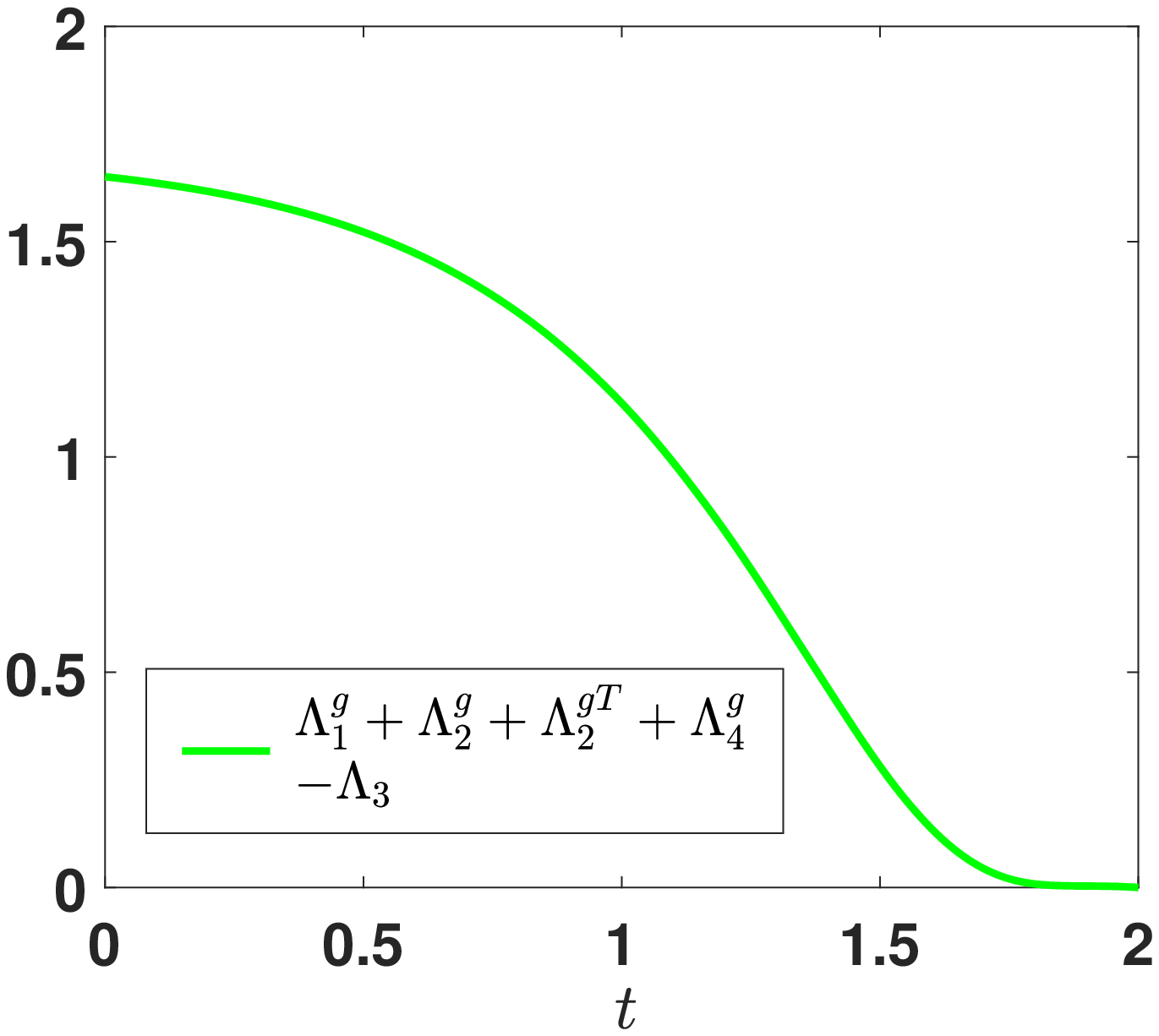, width=6.6cm, height=4cm}
\end{tabular}
\end{center}
\caption{Left panel: Comparison between $\Lambda_1^g + \Lambda_2^g + \Lambda_2^{gT} + \Lambda_4^g$ and $\Lambda_3$.
Right panel: The difference $\Lambda_1^g + \Lambda_2^g + \Lambda_2^{gT} +
\Lambda_4^g - \Lambda_3$. }
\label{fig: Lamg3}
\end{figure}

\section{Conclusion}
\label{sec: conclusion}

This paper studies asymptotic solvability for  LQ mean field social optimization problems with indefinite state and control weight matrices. The analysis involves highly nonlinear large-scale Riccati ODEs due to controlled diffusions. We derive a necessary and sufficient condition for  asymptotic solvability.
We obtain a set of decentralized individual control laws, and further show that its optimality loss  is bounded.
We further check the efficiency gain of the social optimal control with respect to the mean field game.

\appendix

\section{Proof of Lemma~\ref{lm: P submatrices}}
\label{appendix: pf P submatrices}

\begin{proof}
Write the $Nn \times Nn$ identity matrix $I_{Nn}$ as
$I_{Nn} = \diag[I_n, I_n, \cdots, I_n]$.
Let $J_{ij}$ denote the matrix obtained by exchanging the $i$th and $j$th
rows of the submatrices in $I_{Nn}$.
It is easy to check that $J_{ij} = J_{ji}$ and $J_{ij} = J_{ij}^T =
J_{ij}^{-1}$ for all $i, j$.

Denote $\mathbf{P} = ( P_{ij})_{1\leq i, j \leq N}$, where each $P_{ij}$ is an $n\times n$ matrix.
We choose arbitrary $i \neq j$, and denote $J^T_{ij}\mathbf{P} J_{ij} =
\mathbf{P}^\dagger_{(ij)}$. In this proof we write $\mathbf{P}_{(ij)}^\dagger$ as $\mathbf{P}^\dagger$ for simplicity of notation.
We multiply both sides of \eqref{ODE P} from the left by $J_{ij}^T$ and next from the right by $J_{ij}$ to get
\begin{align}
\dot{\mathbf{P}}^\dagger(t) = \mathbf{P}^\dagger \widehat{\mathbf{B}}\left( \mathbf{R} +  2 \mathcal{M}_2(\mathbf{P}^\dag)  \right)^{-1}
 \widehat{\mathbf{B}}^T \mathbf{P}^\dagger -  \mathbf{P}^\dagger
\mathbf{A} - \mathbf{A}^T \mathbf{P}^\dagger - \mathbf{Q} ,  \notag
\end{align}
where we use the following facts with $\Psi = \mathbf{A}$,
$\widehat{\mathbf{B}}$, $\mathbf{R}$, or $\mathbf{Q}$,
\begin{align}
J_{ij}^T ( \mathbf{B}_k \mathbf{e}_k ) J_{ij}
= \begin{cases}  \mathbf{B}_k \mathbf{e}_k & \text{if} \ k \neq i, j , \\
 \mathbf{B}_j \mathbf{e}_j & \text{if} \ k = i , \\
 \mathbf{B}_i \mathbf{e}_i & \text{if} \ k = j ,
\end{cases}
\quad \text{and} \quad
J_{ij}^T \Psi J_{ij} = \Psi .
\notag
\end{align}
Thus $\mathbf{P}^\dagger$ also satisfies \eqref{ODE P}.
It then follows that $J_{ij}^T \mathbf{P} J_{ij} = \mathbf{P}$ for any $ i\ne j$, and the matrix $\mathbf{P}=(P_{ij})_{1\leq i, j \leq N}$ satisfies that
\begin{align}
P_{ii} = P_{jj}, \quad P_{ij} = P_{ji}, \quad P_{ik} = P_{jk}, \quad    P_{ki} = P_{kj} , \quad \forall k \neq i, j.  \notag
\end{align}
This implies that the diagonal submatrices $\{P_{ii}, {1\leq i \leq N}\}$
are equal and all the off-diagonal submatrices $\{P_{ij}, {1\leq i\neq j \leq N}\}$ are equal. Since $\mathbf{P}$ is symmetric, now $P_{ij}=P_{ji}^T=P_{ji}$ for all $i\ne j$.  We denote $P_{ii} = \Pi_1^N$ for all $1\leq i \leq N$, and $P_{ij} = \Pi_2^N$ for all $1\leq i\neq j \leq N$.
Then \eqref{P submatrices} follows.
\end{proof}

\section{Proof of Lemmas~\ref{lm: S submatrices} and \ref{lemma:EN}}
\label{appendix: pf S submatrices}

\begin{proof}[Proof of Lemma \ref{lm: S submatrices}]
Existence and uniqueness holds since \eqref{ODE S} is a linear ODE.
The remaining proof is similar to that of Lemma~\ref{lm: P submatrices}.
We multiply both sides of the ODE \eqref{ODE S} from the left by $J_{ij}$
as in the proof of Lemma~\ref{lm: P submatrices} so that
\begin{align}
  J_{ij}\dot{\mathbf{S}}(t)
& = \mathbf{P} \widehat{\mathbf{B}} ( \mathbf{R} + 2 \mathcal{M}_2(\mathbf{P}) )^{-1} (\widehat{\mathbf{B}}^T J_{ij}\mathbf{S} + \mathcal{M}_1^T(\mathbf{P})) - \mathbf{A}^T  J_{ij}\mathbf{S} .
\notag
\end{align}
Since $J_{ij} \mathbf{S}$ and $\mathbf{S}$ satisfy the same ODE, for arbitrary $i\neq j$, we conclude that $\mathbf{S}$ takes the form \eqref{S submatrices}.
\end{proof}

\begin{proof}[Proof of Lemma \ref{lemma:EN}]
Let $J_{ij}$ be the matrix as defined in the proof of
Lemma~\ref{lm: P submatrices}.
Multiplying both sides of the identity
\begin{align}
 I = ( \mathbf{R} + 2 \mathcal{M}_2(\mathbf{P}) ) ( \mathbf{R}
+ 2 \mathcal{M}_2(\mathbf{P}) )^{-1}    \notag
\end{align}
from the left by $J_{ij}^T$ and next from the right by $J_{ij}$,
for $1\leq i\neq j \leq N$, we obtain
\begin{align}
I & = J_{ij}^T ( \mathbf{R} + 2 \mathcal{M}_2(\mathbf{P}) ) ( \mathbf{R} + 2 \mathcal{M}_2(\mathbf{P}) )^{-1} J_{ij} \notag \\
& = J_{ij}^T ( \mathbf{R} + 2 \mathcal{M}_2(\mathbf{P}) ) J_{ij} J_{ij}^T ( \mathbf{R} + 2 \mathcal{M}_2(\mathbf{P}) )^{-1} J_{ij} . \notag
\end{align}
Since $J_{ij}^T ( \mathbf{R} + 2 \mathcal{M}_2(\mathbf{P}) ) J_{ij}
 =\mathbf{R} + 2 \mathcal{M}_2(\mathbf{P})$, it follows that
$J_{ij}^T ( \mathbf{R} + 2 \mathcal{M}_2(\mathbf{P}) )^{-1} J_{ij} =
( \mathbf{R} + 2 \mathcal{M}_2(\mathbf{P}) )^{-1}$, and thus $E^{N T} =
E^N$. \end{proof}

\section{Mean field game ODEs}
\label{appendix: ODEs Lamg}

The ODEs of $(\Lambda_1^g, \Lambda_2^g, \Lambda_3^g, \Lambda_4^g)$ corresponding to the mean field game in Section~\ref{sec:sub:compSocMFG} are given as follows:
\begin{align*}
& \begin{cases}
   \dot \Lambda_1^g =  \Psi_1(\Lambda_1^g) ,    \\
 \Lambda_1^g (T) =  Q_f , \
 \mathcal{R}_1(\Lambda_1^g(t)) >0, \ \forall t\in[0, T]  ,
\end{cases}   \\ 
& \begin{cases}
 \dot\Lambda_2^g =
 \Lambda_2^g B H^g B^T \Lambda_2^g  + \Lambda_2^g B H^g B^T \Lambda_1^g
+ \Lambda_1^g B H^g B^T \Lambda_2^g    \\
  \hspace{.8cm} - \Lambda_1^g G - \Lambda_2^g (A+G) - A^T \Lambda_2^g
  + Q \Gamma ,    \\
\Lambda_2^g (T) =   - Q_f \Gamma_f  ,
\end{cases} \\
&\begin{cases}
  \dot \Lambda_3^g =     \Lambda_3^g B H^g B^T \Lambda_1^g
  + \Lambda_1^g B H^g B^T \Lambda_3^g
 + \Lambda_4^g  B H^g B^T \Lambda_2^g    \\
\hspace{.8cm} + \Lambda_2^{gT} B H^g B^T ( \Lambda_2^g + \Lambda_4^g) - \Lambda_1^g B H^g B_1^T \Lambda_3^g B_1 H^g B^T \Lambda_1^g         \\
 \hspace{.8cm}
- (\Lambda_1^g + \Lambda_2^{gT} ) B H^g B_0^T(\Lambda_1^g + \Lambda_2^g +
\Lambda_2^{gT} + \Lambda_4^g ) B_0 H^g B^T (\Lambda_1^g + \Lambda_2^g )
 \\
 \hspace{.8cm} - \Lambda_3^g A - ( \Lambda_2^{gT} + \Lambda_4^{gT} ) G
 - A^T \Lambda_3^g - G^T(\Lambda_2^g + \Lambda_4^g )  - \Gamma^T Q \Gamma
, \\
\Lambda_3^g (T) =  \Gamma_f^T Q_f \Gamma_f ,
\end{cases} \\
& \begin{cases}
  \dot \Lambda_4^g =
 \Lambda_4^g B H^g B^T ( \Lambda_1^g   + \Lambda_2^g )
  + \Lambda_1^g B H^g B^T \Lambda_4^g
 + \Lambda_2^{gT} B H^g B^T ( \Lambda_2^g + \Lambda_4^g )   \\
 \hspace{.8cm}
- (\Lambda_1^g + \Lambda_2^{gT} ) B H^g B_0^T
 ( \Lambda_1^g + \Lambda_2^g + \Lambda_2^{gT} + \Lambda_4^g ) B_0 H^g B^T
( \Lambda_1^g + \Lambda_2^g )  \\
  \hspace{.8cm}
- (  \Lambda_2^{gT} + \Lambda_4^g ) G - \Lambda_4^g A
  - G^T (   \Lambda_2^g +  \Lambda_4^g ) - A^T \Lambda_4^g
 - \Gamma^T Q \Gamma  ,    \\
\Lambda_4^g (T) = \Gamma_f^T Q_f \Gamma_f  ,
\end{cases}   
\end{align*}
where we use the notation
$H^g = ( \mathcal{R}_1(\Lambda_1^g) )^{-1}$.

\bibliographystyle{abbrv}
\bibliography{HYmfsocRef}

\end{document}